\documentclass[11pt,a4paper,reqno]{amsart}
\usepackage[hmargin=2.5cm,bmargin=2.5cm,tmargin=3cm]{geometry}
\usepackage{enumerate}

\usepackage{amsmath,amsthm,amssymb,amsfonts,latexsym}

\usepackage[colorlinks, linkcolor=blue, filecolor=blue, citecolor = olive, urlcolor=purple]{hyperref}

\usepackage[german,english]{babel}
\usepackage{orcidlink}

\usepackage[normalem]{ulem}
\usepackage{latexsym}
\usepackage{float,caption}

\usepackage{float}
\usepackage{tikz}					
\usetikzlibrary{decorations.pathreplacing}

\newcommand{\degout}{\deg_{\mathrm{out}}}

\DeclareMathOperator{\supp}{supp}

\DeclareMathOperator{\diam}{diam}

\DeclareMathOperator{\diag}{diag}

\DeclareMathOperator{\aufspan}{span}

\allowdisplaybreaks

\newcommand{\Graph}{{\mathcal G}}

\newcommand{\Ra}{\rightarrow}

\newcommand{\C}{\mathbb{C}}

\newcommand{\N}{\mathbb{N}}

\newcommand{\Z}{\mathbb{Z}}

\usepackage{aliascnt}

\theoremstyle{plain}
\newtheorem{theo}{Theorem}[section]

\newaliascnt{cor}{theo}
\newaliascnt{prop}{theo}
\newaliascnt{lemma}{theo}

\newtheorem{lemma}[lemma]{Lemma}
\newtheorem{prop}[prop]{Proposition}
\newtheorem{cor}[cor]{Corollary}

\aliascntresetthe{cor}
\aliascntresetthe{prop}
\aliascntresetthe{lemma}

\theoremstyle{definition} 

\newaliascnt{defi}{theo}
\newaliascnt{assum}{theo}
\newaliascnt{assums}{theo}
\newaliascnt{prob}{theo}

\newtheorem{defi}[defi]{Definition}
\newtheorem{assum}[assum]{Assumption}

\aliascntresetthe{defi}
\aliascntresetthe{assum}
\aliascntresetthe{assums}
\aliascntresetthe{prob}

\theoremstyle{remark}

\newaliascnt{rems}{theo}
\newaliascnt{rem}{theo}
\newaliascnt{exa}{theo}
\newaliascnt{exs}{theo}

\newtheorem{rem}[rem]{Remark}
\newtheorem{exa}[exa]{Example}

\aliascntresetthe{rems}
\aliascntresetthe{rem}
\aliascntresetthe{exa}
\aliascntresetthe{exs}

\numberwithin{equation}{section}
\numberwithin{figure}{section}
\numberwithin{lemma}{section}

\newcommand{\be}{\begin{equation}}
\newcommand{\ee}{\end{equation}}
\newcommand{\bea}{\begin{eqnarray*}}
	\newcommand{\eea}{\end{eqnarray*}}
\newcommand{\beq}{\begin{eqnarray}}
\newcommand{\eeq}{\end{eqnarray}}

\newcommand{\mG}{\mathsf G}
\newcommand{\mV}{\mathsf V}

\newcommand{\mw}{\mathsf w}
\newcommand{\mv}{\mathsf v}
\newcommand{\mE}{\mathsf E}
\newcommand{\me}{\mathsf e}

\newcommand{\mLn}{\mathsf{L}_n}
\newcommand{\mP}{\mathsf{P}}

 \def\mG{\mathsf{G}}

 \def\mV{\mathsf{V}}
 \def\mU{\mathsf{U}}
 \def\mE{\mathsf{E}}

 \def\mP{\mathsf{P}}
 \def\mK{\mathsf{K}}
 \def\mS{\mathsf{S}}
 \def\mC{\mathsf{C}}
 
 \def\mL{\mathsf{L}}

 \def\mv{\mathsf{v}}
 \def\me{\mathsf{e}}
 \def\mw{\mathsf{w}}
 \def\mz{\mathsf{z}}
 \def\mW{\mathsf{W}}

\DeclareMathOperator{\inr}{inr}

\newcommand{\lSP}{\ell_\nu^2}

\newcommand{\dForm}{\mathcal Q} 
\newcommand{\dLapl}{\mathcal L} 
\newcommand{\dLaplD}{\mathcal L^{\mathrm{D}}_{\mV_0}} \newcommand{\dLaplN}{\mathcal L^{\mathrm{N}}_{\mV_0}} \newcommand{\dLaplB}{\mathcal L^{\mathrm{B}}_{\mV_0}} 
\newcommand{\dFormD}{\mathcal Q^{\mathrm{D}}_{\mV_0}} \newcommand{\dFormN}{\mathcal Q^{\mathrm{N}}_{\mV_0}} \newcommand{\dFormB}{\mathcal Q^{\mathrm{B}}_{\mV_0}} 
\newcommand{\llD}{\lambda_1^{\mathrm{D}}} \newcommand{\llN}{\lambda_2^{\mathrm{N}}} \newcommand{\llB}{\lambda_2^{\mathrm{B}}}

\newcommand{\natExt}[1]{E^{\mathrm{D}}_{#1}}
\newcommand{\neuExt}[1]{E^{\mathrm{N}}_{#1}}

\newcommand{\parti}{{\mathcal P}}

\newcommand{\DirForm}[1]{\dForm^{\mathrm{D}}_{#1}} 
 
\newcommand{\BouForm}[1]{\dForm^{\mathrm{B}}_{#1}} 
 
\newcommand{\NeuForm}[1]{\dForm^{\mathrm{N}}_{#1}}

\newcommand{\benergy[1]}{{\Lambda}^{\mathrm{B}}_{#1}} 

\newcommand{\nenergy[1]}{{\Lambda}^{\mathrm{N}}_{#1}} 
\newcommand{\denergy[1]}{{\Lambda}^{\mathrm{D}}_{#1}} 
\newcommand{\noptenergy[1]}{{\mathfrak L}^{\mathrm{N}}_{#1}} 
\newcommand{\boptenergy[1]}{{\mathfrak L}^{\mathrm{B}}_{#1}} 
\newcommand{\doptenergy[1]}{{\mathfrak L}^{\mathrm{D}}_{#1}} 
\newcommand{\classpart}{\mathfrak P_k} 
\newcommand{\classconn}{\mathfrak C_k}

\title{Spectral minimal partitions of combinatorial graphs}

\subjclass[2020]{49R05, 47B39, 31C20, 47A75, 05C63}

\keywords{Spectral minimal partitions; Laplacians on combinatorial graphs; weighted graphs; infinite graphs; boundary conditions; spectral geometry}

\author[M.~Hofmann]{Matthias Hofmann
\orcidlink{0000-0002-5996-4483}}

\author[J.B.~Kennedy]{James B.\ Kennedy
\orcidlink{0000-0001-5634-0301}}

\author[D.~Mugnolo]{Delio Mugnolo
\orcidlink{0000-0001-9405-0874}}

\author[M.~Plümer]{Marvin Plümer
\orcidlink{0000-0002-7161-4697}}

\address{Matthias Hofmann, Lehrstuhl Numerische Mathematik, Fakult\"at Mathematik und Informatik, Fern\-Universit\"at in Hagen, D-58084 Hagen, Germany}
\email{matthias.hofmann@fernuni-hagen.de}

\address{James B. Kennedy, Departament of Mathematics, University of Aveiro, 3810-193 Aveiro, Portugal}
\email{jbkennedy@ua.pt}

\address{Delio Mugnolo, Lehrstuhl Analysis, Fakult\"at Mathematik und Informatik, Fern\-Universit\"at in Hagen, D-58084 Hagen, Germany}
\email{delio.mugnolo@fernuni-hagen.de}

\address{Marvin Plümer, Lehrstuhl Analysis, Fakult\"at Mathematik und Informatik, Fern\-Universit\"at in Hagen, D-58084 Hagen, Germany}
\email{marvin.pluemer@fernuni-hagen.de}

\date{\today}

\thanks{The authors are grateful to Gregory Berkolaiko (Texas A\&M) for helpful discussions.\\
This work was supported by the Funda\c{c}\~ao para a Ci\^encia e a Tecnologia (FCT), Portugal, within the scope of the project Spectral Optimal Partitions: geometric and numerical analysis, reference \href{https://doi.org/10.54499/2023.13921.PEX}{2023.13921.PEX} (M.H. and J.B.K.), and via the research centers CIDMA, references \href{https://doi.org/10.54499/UID/04106/2025}{UID/04106/2025} and \href{https://doi.org/10.54499/UID/PRR/04106/2025}{UID/PRR/04106/2025}, \url{} (J.B.K.) and GFM, references \href{https://doi.org/10.54499/UID/00208/2025}{UID/00208/2025} and \href{https://doi.org/10.54499/UID/PRR/00208/2025}{UID/PRR/00208/2025} (J.B.K.) under the FCT Multi-Annual Financing Program for R{\&}D Units. 
D.M.\ and M.P.\ were partially supported by the Deutsche Forschungsgemeinschaft (Grant 397230547). \\
 This article is based upon work
 from COST Actions 18232 MAT-DYN-NET and 24122 mSPACE, supported by COST (European Cooperation in Science and
 Technology), \url{www.cost.eu}.
}

\begin{document}

\begin{abstract}
This paper investigates spectral minimal partitions for weighted graphs, thus extending the extensive class of results that are currently available on domains and, to a lesser extent, manifolds and metric graphs.
We provide a rigorous framework for analyzing graph Laplacians under Dirichlet, Neumann, and boundaryless energy formulations; a central focus of the study is establishing existence theorems for minimal partitions. While existence is straightforward for finite connected graphs due to the finiteness of the class of admissible partitions, infinite graphs require advanced topological and functional-analytic machinery. Specifically, we introduce the notion of canonical compactifiability, which relates to compact embeddings and uniform Poincar\'e-type constants for Neumann and boundaryless energies; and an appropriate notion of subgraph convergence. In this way, we can relax the spectral minimal problem on infinite graphs by reducing it to the study of finite graphs; and can, thus, guarantee that optimal spectral energies are actually attained by appropriate partitions even in non-compact settings.
\end{abstract}

\maketitle

\section{Introduction}
This article introduces a systematic study of spectral minimal partitions of combinatorial graphs.

The isoperimetric constant of a closed compact manifold was introduced by Cheeger in~\cite{Che70}, and it was shown in~\cite{Yau75} that it can be interpreted as a variational quantity, as the lowest positive critical value of the Rayleigh quotient associated with the $p$-Laplacian, for $p=1$. The same idea was extended to graphs in~\cite{Dod84,AloMil85} and later in \cite{BuhHei09,SzlBre10}, starting with $p=2$ and then for $p\to 1$, paving the road to the development of spectral clustering. The main issue with this method is that the nodal domains of the eigenfunctions associated with higher eigenvalues cannot guarantee the delivery of a sought-after number of clusters. Higher-order Cheeger partitioning via the $p$-Laplacian was finally achieved in~\cite{LeeOveTre12} for $p=2$, although this technique does not seamlessly carry over to the case of $p\ne 2$, see \cite{HeiLenMug15}. Here we follow another approach, strongly inspired by the theory of spectral minimal partitions on domains \cite{HelHofTer09} and metric graphs~\cite{KenKurLen21,HofKenMug21}.

{Roughly, by a partition of a graph $\mG$ we will understand a family of connected, pairwise disjoint subgraphs of $\mG$ (which are usually called \emph{clusters}, or \emph{cells} in the literature), whose union is equal to $\mG$, that is, the partition should be \emph{exhaustive}. Other definitions are possible; for example, we may not require that the clusters be connected, or we may allow non-exhaustive partitions, but to keep things simpler we will only study the baseline case of exhaustive partitions consisting of connected clusters.}

 The basic idea of a spectral minimal partition was introduced in~\cite{ConTerVer05}, in the context of Euclidean domains: it boils down to introducing a functional 
\[
\Lambda_p(\mathcal P):=\left(\sum_{i=1}^k |\mu(\Omega_i)|^p\right)^\frac{1}{p}
\]
defined on an appropriate class $\mathfrak P$ of all partitions of a given spatial structure into $k$ disjoint \textit{clusters} $\Omega_1,\ldots,\Omega_k$, where $\mu(\Omega_i)$ denotes the first strictly positive eigenvalue of {a suitable} Laplacian on $\Omega_i$ -- the \textit{energy} of $\Omega_i$; and then to take its infimum over $\mathfrak P$.
Akin to domains and metric graphs, it is then crucial to specify two requirements:
\begin{itemize}
\item the class of partitions of interest;
\item  the relevant realisation of the Laplacian to be considered on each $\Omega_i$.
\end{itemize}
Our main goal here is to develop an existence theory for such spectral minimal partitions in the context of combinatorial graphs, which may be finite or countably infinite. We will consider three natural realisations of the Laplacian, each defining a different notion of energy: the \textit{Dirichlet energy}, the \textit{Neumann energy} and the \textit{boundaryless energy}: the latter one is induced by the Laplacian on the \textit{graph} $\Omega_i$, the other two reflect the nature of $\Omega_i$ as a \textit{subgraph} of $\mG$ and are, therefore, arguably more geometric; they are induced by two Laplacian realisations that, to the best of our knowledge, were introduced in~\cite{Fri93,ChuGraYau96}, respectively. We review all of them in Section~\ref{sec:some-examples}. Robin-type realisations have been studied in~\cite{Fri93,BenCarEnc00,BerRazSmi12}, among others, but we will not delve into this generalisation.

A comparable theory was also developed in~\cite{OstWhiOud14}, albeit for the Dirichlet energy, and finite graphs only. On the other hand, the authors of~\cite{OstWhiOud14} relax the notion of graph partition to allow for clusters that actually are partitions of unity over the vertex set, see \cite[Section~3]{OstWhiOud14} (clearly, each graph partition defines a partition of unity on $\mV$, but not vice versa).

Also due to advances in the theory of infinite graphs over the last ten years or so \cite{KelLenWoj21}, including the study of different realisations of the Laplacian, it is now easy to find natural conditions that ensure that, even on infinite graphs, these realisations of the Laplacian have discrete spectrum and, in particular, the lowest positive eigenvalue $\mu$ is a well-defined quantity. (Cheeger inequalities leading to a $ k$-partitioning for $k=2$ have already been extensively studied for infinite combinatorial graphs, too, see \cite{KelMug16} and references therein; a different approach based on iso\textit{capacitary} constants, leading again to a 2-partitioning, has been recently proposed in~\cite{HuaWan25}.)

In \cite{BerRazSmi12} a connection between the nodal domains of graph Laplacian eigenvectors and criticality of graph partitions with respect to a spectral energy functional were explored, and a one-to-one correspondence between critical graph partitions and the signed Laplacian was recently established in \cite{Men26}.  

We also mention that \cite{BeHo25} explored a similar idea to investigate the correspondence between eigenvalues of the $p$-Laplacian and critical points of a family of eigenvalue problems on cut graphs. The issue of spectral minimal $k$-partitions induced by the discrete $p$-Laplacian {(again for finite graphs and a Dirichlet realisation)} was discussed in~\cite[Section~9.1]{DeiTudZha25}, with a special focus on the limiting case $p\to \infty$.

Let us present a plan of the present work: In \autoref{sec:some-examples} we fix the notation, introduce the underlying function spaces and energy functionals and develop the spectral theory. We also give an example to provide intuition with regard the spectral energies in the partition problems that we consider.  In \autoref{sec:smps} we develop the existence theory for the types of spectral minimal partition problems under consideration. In \autoref{sec:spec-ineq}, we discuss some spectral inequalities for the associated partition energies of these spectral minimal partitions. We conclude, {in \autoref{sec:snakes-and-ladders}, with a prototypical family of examples, consisting of finite and infinite ladder graphs. Here our goals are twofold: for finite ladders, we show how the different realizations of the Laplacian lead to profoundly different types of minimizing partitions; for infinite ladders, we show why auxiliary assumptions on the graphs, in particular as regards vertex and edge weights, is necessary in order to guarantee that minimizing partitions exist.}

\section{Preliminaries and Notation}\label{sec:some-examples}

\subsection{Metric structure and function spaces on infinite graphs}
Let $\mG$ be a finite or countably infinite graph with vertex set $\mV$ and edge set $\mE$; we refer to~\cite[Chapt.~8]{Die05} for general definitions. We will always assume, without further comment, that \(\mG\) is \emph{simple} and \emph{connected}, i.e.\ \(\mG\) has no loops or parallel edges and any two vertices in \(\mG\) can be connected with a path in \(\mE\). Because \(\mG\) is simple, each edge \(\me\) can be identified with a unique two-element set \(\{\mv,\mw\}\) of distinct vertices \(\mv,\mw\in\mV\). For \(\mv\) let \(\mE_\mv\) denote the set of edges incident to \(\mv\). 

We endow each vertex $\mv\in\mV$ and each edge $\me\in\mE$ with weights

$\nu(\mv)\in (0,\infty)$ and $\omega(\me)\in (0,\infty)$. We require that the weights are locally finite, i.e.,
\begin{equation}\label{eq:omega-finite}
\sum_{\me\in\mE_\mv}\omega(\me)<\infty
\end{equation}
holds for all vertices \(\mv\in\mV\).
Furthermore, we shall extend \(\nu\) to a discrete measure on the vertex set \(\mV\) writing
\[
	\nu(\mS):=\sum_{\mv\in \mS}\nu(\mv)
\]
for subsets \(\mS\subset\mV\); and likewise for $\omega$.

We say that \(\nu\) is \emph{finite} if
\begin{equation}
\label{eq:finite-nu}
	\nu(\mV)<\infty.
\end{equation}

Usually the edge weight \(\omega(\me)\) is interpreted as an interaction strength between the incident vertices \(\mv\) and \(\mw\). (This interpretation immediately becomes clear if one looks at the quadratic form \(\dForm\) defined below: if the edge weight \(\omega(\me)\) becomes larger, the values \(f(\mv)\) and \(f(\mw)\) need to be closer for \(\dForm(f)\) to be small.) One may, e.g. think of \({1}/{\omega(\me)}\) as the length of the edge \(\me\) like in \cite{GeoHaeKel15, KosMalNei17}. We define a metric on \(\mV\), accordingly: for a path \(\gamma=(\mv_0,\mv_1\ldots,\mv_n)\), where any two successive vertices \(\mv_{j-1}\) and \(\mv_j\) are connected by an edge \(\me_j\), we define its length by
	\[L(\gamma):=\sum_{j=1}^n\frac{1}{\omega(\me_j)}.\]
On \(\mV\) we consider the \emph{shortest path metric} given by
\begin{equation}\label{eq:distfunction}
    d(\mv,\mw):=\inf\{L(\gamma)~|~\gamma\text{ is a path connecting }\mw\text{ and }\mv\},\quad \mv,\mw\in\mV.
\end{equation}

The diameter of \(\mG\) with respect to the metric \(d\) is, as usual,
	\[\diam_d(\mG):=\sup\{d(\mv,\mw)~|~\mv,\mw\in\mV\}.\]

In this way, $\mG$ is a  metric measure space in the sense of~\cite{Stu06}. For $1 \leq p < \infty$, let

\[
	\ell_\nu^p(\mV):=\left\{f:\mV\Ra \C~|~\sum_{\mv\in\mV}\nu(\mv)|f(\mv)|^p<\infty\right\}
\]
be the space of $p$-summable functions on $\mV$, a Banach space for the natural norm. When $p=2$ this is a Hilbert space when equipped with the scalar product given by
\[
	(f,g)_{\lSP(\mV)}=\sum_{\mv\in\mV}\nu(\mv)f(\mv)\overline{g(\mv)}
\]
and its induced norm \(\|\cdot\|_{\lSP(\mV)}\). Moreover, let \(\ell^\infty(\mV)\) denote the space of bounded functions on \(\mV\) equipped with the norm given by
	\[\|f\|_{\ell^\infty(\mV)}=\sup_{\mv\in\mV}|f(\mv)|.\]

\begin{lemma}\label{prop:comp-embedd}
	If \(\nu\) is finite, then \(\ell^\infty(\mV)\) is compactly embedded in \(\ell^p_\nu(\mV)\) for any $p\in [1,\infty)$.
\end{lemma}

\begin{proof}
	First of all, \(\ell^\infty(\mV)\subset\ell^p_\nu(\mV)\) holds, because \(\nu\) is finite,  Now, consider a sequence \((\mV^n)_{n\in\N}\) of finite subsets \(\mV^n\subset \mV\) with \(\mV=\bigcup_{n\in \N}\mV^n\) and let \(I_n:\ell^\infty(\mV)\rightarrow\ell^p_\nu(\mV)\) denote the cut-off operator given by
		\[(I_nf)(\mv)=\begin{cases} f(\mv), &\mv\in\mV^n,\\ 0, &\mv\notin\mV^n. \end{cases}\]
	Each \(I_n\) is of finite rank and the sequence \((I_n)_{n\in\N}\) converges to the embedding \(I:\ell^\infty(\mV)\rightarrow\ell^p_\nu(\mV)\) with respect to the operator norm, thus \(I\) is compact.
\end{proof}

\subsection{Laplacians on infinite graphs}\label{subsec:lapl}

{

With the goal of partitioning a given graph into several clusters, it is natural to use nodal domains of some self-adjoint realisation of the \emph{Laplace matrix}. We will define these operators via Dirichlet forms on combinatorial graphs. We refer to \cite{KelLenWoj21} for an introduction to the topic. Let us briefly recall some basic facts about the operator. 
}

We consider the sesquilinear form \(\dForm\) given by
\begin{equation}
\label{eq:in-form}
\begin{split}	\dForm(f,g):=&\left(\mathcal I^\top f,\mathcal I^\top g\right)_{\ell^2_\omega(\mE)}\\
=&\sum_{\me=\{\mv,\mw\}\in\mE}\omega(\me)(f(\mw)-f(\mv))\overline{(g(\mw)-g(\mv))}
\end{split}
\end{equation}
	
with maximal domain
	\[
	h^1(\mV):=\left\{f\in\lSP(\mV)~|~\sum_{\me=\{\mv,\mw\}\in\mE}\omega(\me)|f(\mw)-f(\mv)|^2 <\infty\right\}.\]
	(Here $\mathcal I$ is the signed incidence matrix of an arbitrary directed version of $\mG$, see~\cite[Chapter~2]{Mug14}.)
Then \( \left(\dForm,h^1(\mV)\right)\) is a densely defined, closed quadratic form, cf.~\cite[Section~1.1]{KelLenWoj21}.
We denote by \((\cdot,\cdot)_\dForm\) the scalar product on \(h^1(\mV)\) induced by \(\dForm\), i.e.\
	\[(f,g)_\dForm=(f,g)_{\lSP(\mV)}+\dForm(f,g).\]
In particular, $h^1(\mV)$

is a Hilbert space. Note that  $h^1(\mV)$ depends on both the weight functions $\nu$ and $\omega$; however, since we will always consider these to be fixed throughout the paper, we will suppress this dependence in the notation.

We naturally regard $h^1(\mV)$ as a discrete version of a prototypical first-order Sobolev space; indeed, in analogy with the theory of Sobolev spaces on Euclidean domains, we also consider the space obtained closing a suitable space of test functions with respect to $\|\cdot\|_{\dForm}$,
\[
h^1_0(\mV):= \overline{c_{00}(\mV)}^{_{\|\cdot\|_{\mathcal Q}}},
\]
where $c_{00}(\mV)$ denotes the set of functions $f: \mV \to \C$ which have finite support, i.e., are nonzero on only finitely many vertices.

{
In the following, we will fix a form domain $E=D(\dForm)$ with the following properties:

\begin{assum}
\label{assum:domform}
The set $E=D(\dForm)$ is a closed subspace of $h^1(\mV)$ such that
\begin{equation}
\label{eq:domform}
h_0^1(\mV) \subset D(\dForm) \subset h^1(\mV). 
\end{equation}
\end{assum}

Under this assumption, we} consider the self-adjoint operator 
\(
\dLapl
\) 
on $\ell^2_\nu(\mV)$ associated with the (closed) symmetric form $\left(\dForm,D(\dForm) \right)$; then $\dLapl$ acts on test functions as the formal Laplacian, i.e.,
\[
{\mathcal L}f(\mv)=\frac{1}{\nu(\mv)}\sum_{\substack{\me=\{\mv,\mw\}\in\mE_\mv \\ \mw\in\mV_0}}\omega(\me)(f(\mv)-f(\mw))=0\qquad \hbox{for all } f\in c_{00}(\mV),\ \mv\in \mV,
\]
and the Friedrichs extension of $\mathcal L_{|c_{00}(\mV)}$ is precisely the operator on $\ell^2_\nu(\mV)$ associated with $\left(\dForm,h^1_0(\mV)\right)$.

In some relevant cases (e.g., for \textit{canonically compactifiable graphs} as discussed in~\cite{GeoHaeKel15}, see \autoref{def:canonically-compactifiable} below), a notion of \emph{boundary at infinity} (technically speaking, a so-called \emph{Royden boundary}) can be defined on which boundary conditions must be imposed for the resulting Laplacian to be self-adjoint: by taking \autoref{assum:domform} and, in particular, \eqref{eq:domform}, we can freely handle various kind of conditions, including mixed conditions of Dirichlet/Neumann-type at the (Royden) boundary of the graph.

For further details on Laplacians on infinite weighted graphs we refer to \cite{HaeKelLen12}, where the question of whether these realizations coincide is thoroughly discussed. If \(\mV\) is infinite the Laplacian does not need to have discrete spectrum. A precise characterization of pure discreteness of a Laplacian spectrum on sparse graphs (including trees) is  known \cite[Theorem~1.2]{BonGolKel15}.
The following proposition will give a further useful condition.

\begin{prop}\label{prop:comp-embed-formdomain-ultra}
	If \(\nu\) is finite and $h^1(\mV)$ is continuously embedded into $\ell^p_\nu(\mV)$ for some $2<p\leq\infty$, then \(h^1(\mV)\) is compactly embedded in \(\lSP(\mV)\). In particular, {for any choice of domain $D(\dForm)$ satisfying \autoref{assum:domform}, the associated operator} \(\dLapl\) has pure discrete spectrum.
\end{prop}

\begin{proof}
It is sufficient to consider the operator on $\ell^2_\nu(\mV)$ associated with $\left(\dForm,h^1(\mV)\right)$, which is a Dirichlet form by~\cite[Proposition 1.14]{KelLenWoj21}. Hence, the associated semigroup {$(e^{-t\mathcal{L}})_{t\ge 0}$} is Markovian. Then, because $c_{00}(\mV)$ and hence $h^1(\mV)\cap \ell^1_\nu(\mV)$ are clearly dense in $\ell^1(\mV)$, \cite[Theorem in Section~7.3.2]{Are04} implies that ultracontractivity of the semigroup -- hence, in particular, the fact that it maps $\ell^1_\nu(\mV)$ to $\ell^\infty(\mV)$ -- is equivalent to a Sobolev-type embedding of $h^1(\mV)$ into $\ell^p_\nu(\mV)$ for some $p>2$. But the Dunford--Pettis--Kantorovich--Vulikh principle states that the bounded operators from $\ell^1_\nu(\mV)$ to $\ell^\infty(\mV)$ are precisely those with an $\ell^\infty(\mV\times \mV)$-kernel. Because $\nu$ is finite, $\ell^\infty(\mV\times \mV)\hookrightarrow \ell^2_{\nu\times \nu}(\mV\times \mV)$, hence {for each $t>0$, the operator $e^{-t\mathcal{L}}$} is Hilbert--Schmidt, hence of trace class and, in particular, compact; we finally conclude that the form domain is compactly embedded into $\ell^2_\nu(\mV)$.
\end{proof}
We next recall a sufficient condition for the continuous embedding of $h^1(\mV)$ into {$\ell^\infty(\mV)$ and hence (by \eqref{eq:finite-nu}) into} $\ell^p_\nu(\mV)$.

\begin{defi}
\label{def:canonically-compactifiable}
	The graph \(\mG\) is called \emph{canonically compactifiable} if $h^1(\mV)$ is continuously embedded into $ \ell^\infty(\mV)$.

\end{defi}

A well-known example of a canonically compactifiable graph is $\Z^d$ for any $d\in \N\setminus\{2\}$, see \cite{Var89} or \cite[Section~4.A]{Woe00}.

\begin{rem}
By \cite[Theorem~5.8]{GeoHaeKel15}, if $\mG$ is canonically compactifiable, then the domain of the operator associated with $\left(\dForm,h^1_0(\mV) \right)$ consists of all functions $f\in h^1_0(\mV)$ such that $\mathcal Lf\in \ell^2_\nu(\mV)$ and such that $f$ satisfies Dirichlet conditions at the (Royden) boundary of $\mG$.
\end{rem}
The following is now an immediate consequence of \autoref{prop:comp-embed-formdomain-ultra}; it was already stated in~\cite[Lemma~1.1]{LenSchSto18} (with a different proof).

\begin{cor}\label{cor:comp-embed-formdomain}
	If \(\nu\) is finite and \(\mG\) is canonically compactifiable, then \(h^1(\mV)\) is compactly embedded in \(\lSP(\mV)\). In particular, {for any choice of $D(\dForm)$,} the operator \(\dLapl\) has pure discrete spectrum.
\end{cor}

 It is notable that the property of canonical compactifiability, which a priori is purely analytic, can be enforced by geometric properties. Let us recall the following, which follows immediately from \cite[Lemma~4.4]{GeoHaeKel15}.

\begin{lemma}\label{lem:finfingeohaekel}
If $\nu$ is finite, then a connected graph \(\mG\) is canonically compactifiable if it has finite diameter.
\end{lemma}

In particular, $\mG$ has finite diameter (with respect to $d$ as defined in \eqref{eq:distfunction}) in the following two cases, see~\cite[Examples~4.6 and 4.7]{GeoHaeKel15}:
\begin{itemize}
\item $\omega^{-1}\in \ell^1(\mE)$;
\item $\mG$ is a rooted tree and all finite paths starting at the root have uniformly bounded length.
\end{itemize}

\begin{rem}\label{rem:no-subg}
\begin{enumerate}
\item\label{item:no-subg}
\autoref{def:canonically-compactifiable} corresponds to the terminology used in \cite[Section~4.2]{HuaKelSch23} and will be sufficient for our purposes. Let us however remark that, in earlier literature, the term \emph{canonically compactifiable} sometimes  referred to the (stronger) condition that the space
	\[
\mathcal E(\mV):=\left\{f:\mV\to\C~|~\sum_{\me=\{\mv,\mw\}\in\mE}\omega(\me)|f(\mw)-f(\mv)|^2 <\infty\right\}
	\]
of finite energy functions is contained in $\ell^\infty(\mV)$ (as in, e.g., \cite{GeoHaeKel15}). This is equivalent to the validity of certain \emph{Poincaré-type inequalities} as in \cite[Section~11.1]{Chu97} or \cite[Section~4.1]{GeoHaeKel15} (even for locally infinite graphs, see \cite[Example~8.4]{GeoHaeKel15}.
For characterizations of this stronger condition in terms of intrinsic measures and of the algebra property of the space of finite energy, see \cite[Section~4.1]{GeoHaeKel15} or \cite[Section 4]{Puc21} and in particular \cite[Section~3]{LenSchSto18}, where the assumption of local finiteness~\eqref{eq:omega-finite} is removed.

\item\label{item:no-finite-diam} If $\omega\equiv 1$ and $\nu\equiv 1$, an infinite graph has neither finite diameter nor finite measure, then neither \autoref{prop:comp-embed-formdomain-ultra} nor \autoref{cor:comp-embed-formdomain} can be applied. However, it is known that the Laplacian (whose restriction to $c_{00}(\mV)$ is well-known to be essentially self-adjoint under these assumptions on $\omega$ and $\nu$) can still have pure point spectrum: this is true

on general $k$-sparse graphs (see \cite[Definition~1.1]{BonGolKel15}) provided $\liminf_{j\to\infty} \deg(\mv_j)=\infty$, for an arbitrary enumeration of the vertex set, see also \cite[Corollary~12]{Mel17} for a stronger assertion on radially symmetric trees. (Several extensions of the results from \cite{BonGolKel15} to weighted graphs can be found in \cite[Chapter~10]{KelLenWoj21}.)

\item\label{item:mms-discrete}
{
Sufficient conditions for a  compact embedding of $h^1(\mV)$ into $\ell^2_\nu(\mV)$ that rely upon an abstract Kolmogorov--Riesz criterion -- and, thus, circumvent factorizing through $\ell^\infty(\mV)$ or $\ell^p_\nu(\mV)$ -- are known in the theory of metric measure spaces, too. In particular, it is known that $h^1(\mV)$ is compactly embedded into $\ell^2_\nu(\mV)$ whenever  the conditions
\begin{itemize}
\item  $\mG$ has finite diameter,
\item $0<\nu(B)<\infty$ for all balls $B$ in $\mG$,
\item each ball of radius $r$ in $\mG$ can be covered by at most $M$ balls of radius $\frac{r}{2}$, where $M$ is independent of $r$,
\end{itemize}
are all satisfied, see \cite[Proposition~2.1]{BjoKal22}. Similar sufficient and necessary conditions for the embedding of $h^1(\mV)$ into $\ell^p_\nu(\mV)$, needed to apply \autoref{prop:comp-embed-formdomain-ultra}, are given in~\cite[Theorem~1.3]{BjoKal22}.
}

\item\label{item:metric-gr} 
In the case of infinite \emph{metric} graphs of finite measure, the issue of compact embedding of Sobolev spaces in $L^2$-spaces was studied in \cite{DufKenMug25}; {a one-parameter family of diagonal comb metric graphs $(\Graph_\alpha)_{\alpha>0}$ was presented in~\cite[Section~3.2]{DufKenMug25} for which $H^1_0(\Graph_\alpha)$ is compactly embedded into $L^2(\Graph_\alpha)$ for all $\alpha>0$, but $H^1(\Graph_\alpha)$ is compactly embedded into $L^2(\Graph_\alpha)$ only for $\alpha>\frac12$. The same example carries over to the discrete setting; on the one hand, for the corresponding one-parameter family of diagonal comb \textit{weighted} graphs $(\mG_\alpha)_{\alpha>0}$, the compact embedding of $h^1_0(\mV_\alpha)$ into $\ell^2_\nu(\mV_0)$ follows from \cite[Lemma~2.10]{KosNic19}. On the other hand, the argument of \cite[Theorem~3.4.(2)]{DufKenMug25} can be transferred to the discrete case upon considering the combinatorial graph $\mG_\alpha$ underlying $\Graph_\alpha$: more precisely. the sequence of tent functions on $\Graph_\alpha$ constructed in the proof of \cite[Theorem~3.4.(2)]{DufKenMug25} naturally induces a sequence of orthonormal functions in $\ell^2(\mV_\alpha)$ that is equibounded in $h^1(\mV_\alpha)$.} 
\end{enumerate}
\end{rem}

\subsection{Laplacians on subgraphs and their eigenvalues}

We also want to consider suitable conditions at the boundary of suitable subgraphs (the candidate clusters in our partitions): this will be done in this section, extending some ideas in~\cite{Fri93} and~\cite[Chapter 8]{Chu97} to the case of infinite subgraphs.

Given a weighted graph $\mG=(\mV,\mE,\nu,\omega)$, let $\mV_0$ be a subset of $\mV$: we denote by $|\mV_0|$ and $|\mV|$ the number of their elements, respectively.
The \emph{induced subgraph} $\mG_0=(\mV_0,\mE_0,\nu_0,\omega_0)$ is defined in a natural way: its vertex set is $\mV_0$, two vertices in $\mV_0$ are adjacent in $\mG_0$ if and only if they were adjacent in $\mG$, and the weights $\nu_0,\omega_0$ are the restrictions of $ \nu,\omega$. In an abuse of terminology we will sometimes say that $\mV_0$ is connected if the subgraph $\mG_0$ induced by $\mV_0$ is connected with respect to the canonical topology on $\mG$, i.e., the topology induced by the shortest path metric in \eqref{eq:distfunction}.
 
To avoid trivialities, in the following we assume throughout that $\mV_0\subsetneq \mV$
and denote by $\mE_0$ the edge set of the subgraph of $\mG$ induced by $\mV_0$. We denote by $\partial \mV_0$ the \emph{edge boundary} of $\mV_0$, i.e., the (nonempty, since $\mV\setminus \mV_0 \ne \emptyset$ and $\mG$ is connected) set of edges between elements of $\mV_0$ and elements of $\mV\setminus \mV_0$; and by $\delta \mV_0$ the \emph{vertex boundary} of $\mV_0$, i.e., the set of vertices in $\mG$ that are not elements of $\mV_0$ but are adjacent to at least one vertex in $\mV_0$.
The formal  Laplacian then takes the block structure
\begin{equation}\label{eq:block-dlapl}
\dLapl \equiv
\begin{pmatrix}
\dLapl_{\mV_0, \mV_0} & \dLapl_{\mV_0, \delta\mV_0} & 0 \\
\dLapl_{\delta \mV_0, \mV_0} & \dLapl_{\delta\mV_0, \delta\mV_0} & \dLapl_{\delta\mV_0, \widetilde{\mV_0}} \\
0 & \dLapl_{\widetilde{\mV_0}, \delta\mV_0} & \dLapl_{\widetilde{\mV_0}, \widetilde{\mV_0}}\end{pmatrix}
\end{equation}
with respect to the partition
\[
\mV=\mV_0\sqcup \delta \mV_0 \sqcup\widetilde{\mV_0} \]
where
$\widetilde{\mV_0}:=\mV\setminus(\mV_0\cup\delta \mV_0)$.

\subsubsection{Dirichlet energy}\label{sec:dirichlet-energy}
We want to study Laplacians acting on functions that satisfy
\emph{Dirichlet conditions at the vertex boundary} $\delta$, i.e. 
\begin{equation}\label{eq:diri-bc}
f(\mv)=0\qquad \hbox{for all }\mv\in \delta \mV_0.
\end{equation}

 Let $\natExt{\mV_0} : \C^{\mV_0} \to \C^\mV$ be the \textit{natural extension operator}, which extends functions supported on $\mV_0$ by 0 to the whole set $\mV$; its adjoint ${\natExt{\mV_0}}^{\! *} : \C^\mV \to \C^\mV_0$ is the restriction operator. 
 
 {We consider the quadratic form
 \[
 \dFormD(f) := \dForm(\natExt{\mV_0}f)
 \]
 (where here and throughout we write $\dForm(g):=\dForm(g,g)$); the corresponding form domain,
 \begin{equation}\label{eq:domdirform}
 D(\dFormD) = \{f \in \ell^2_{\nu_0}(\mV_0)~|~\natExt{\mV_0}f \in D(\dForm)\}= \{ f \in h^1(\mV_0)~|~\natExt{\mV_0}f \in D(\dForm)\},
 \end{equation}
 a subspace of $h^1(\mV_0)$, can be canonically identified with
 \begin{displaymath}
     \{ f \in D(\dForm)~|~f(\mv)=0 \text{ for all } \mv \in \mV \setminus \mV_0\}
     = \{ f \in D(\dForm)~|~\supp f \subset \mV_0\},
 \end{displaymath}
 which is seen to be a closed subspace of $D(\dForm)$ and hence of $h^1(\mV)$, via the isometry $\natExt{\mV_0}$ between the sets. We refer to the self-adjoint operator associated with $\dForm$ as the  Laplacian with Dirichlet conditions at $\delta\mV_0$, which is the (possibly unbounded) operator on $\ell^2_\nu(\mV_0)$ formally given by
 \[
 \dLaplD :={\natExt{\mV_0}}^{\! *} \dLapl \natExt{\mV_0}=\mathcal L_{\mV_0,\mV_0}
 \]
(with respect to the block decomposition in \eqref{eq:block-dlapl}).
 }

\begin{lemma}\label{lem:dlapld-selfadj}
Let $\mV_0\subsetneq \mV$. Then the following assertions hold. 
\begin{enumerate}
\item\label{item:monotonicity-ppsp} If $\dLapl$ has compact resolvent (and thus pure point spectrum), then so does $\dLaplD$.
\item\label{item:monotonicity-dir} Let $\mV'_0\subset \mV_0$. If $\dLapl$ has {compact resolvent}, then the 

eigenvalues of $\dLaplD$ and $\mathcal L^{\mathrm{D}}_{\mV'_0}$ satisfy
\[
\lambda_k(\mV_0)\le \lambda_k(\mV'_0),\qquad k=1,2,\ldots.
\]
\end{enumerate}
 \end{lemma}

We recall that, by standard theory, $\dLapl$ has compact resolvent if and only if the embedding of the form domain $D(\dForm)$ into $\ell^2_\nu(\mV)$ is compact.

In the case of finite graphs, the eigenvalue monotonicity property (3) was observed in~\cite[Theorem~2.3]{Fri93}.

 \begin{proof}

(1) The assumption that $\dLapl$ has compact resolvent is equivalent to $D(\dForm)$ being compactly embedded into $\ell^2_\nu(\mV)$.  Let $(f_n)_{n\in \N}$ be a sequence in the unit ball of $D(\dFormD)$. Therefore,
\[
\|\natExt{\mV_0}f_n\|_{\ell^2_\nu(\mV)}
+\dForm(\natExt{\mV_0}f_n)\le 1\qquad \hbox{for all }n\in \N,
\]
i.e., $(\natExt{\mV_0}f_n)_{n\in \N}$ lies in the unit ball of $D(\dForm)$, hence by compactness it has a subsequence $(\natExt{\mV_0}f_{n_k})_{k\in \N}$ that converges in $\ell^2_\nu(\mV)$; because $\natExt{\mV_0}$ is an isometry, $(f_{n_k})_{n\in \N}$ must be convergent in $\ell^2_{\nu_0}(\mV_0)$, too: i.e., $D(\dFormD)$ is compactly embedded into $\ell^2_{\nu_0}(\mV_0)$ and the assertion follows.

(2) The assertion follows immediately from the Courant--Fischer min-max principle, since the image under $\natExt{\mV_0}$ of any $k$-dimensional subspace of $D(\dFormD)$ is a $k$-dimensional subspace of $D(\dForm)$.
\end{proof}

We call $\dFormD$ the \emph{Dirichlet energy form} (with respect to $\mV_0$).
A direct computation shows that
\begin{equation}\label{eq:dformd-as-schroed}
\begin{split}    
\dFormD(g)&=\|{\mathcal I}^\top \natExt{\mV_0}g\|^2_{\ell^2_\omega(\mE)}\\
&=\sum_{\me=\{\mv,\mw\}\in\mE_0}\omega(\me)|g(\mv)-g(\mw)|^2+\sum_{\mw\in\mV_0}\degout(\mw)|g(\mw)|^2
\end{split}
\end{equation}
for all $g\in D(\dFormD)$,

where we recall that $\mathcal I$ denotes the signed incidence matrix of $\mG$ (cf. \eqref{eq:in-form}), which we interpret as an operator from $\C^\mV$ to $\C^\mE$, and where \(\degout:\mV_0\rightarrow [0,\infty)\) is 

given by
\begin{equation}\label{eq:degout-def}
\degout(\mw):=(
-\mathcal L_{\mV_0,\delta\mV_0}{\mathbf{1}_{\delta\mV_0}})(\mw)=
\frac{1}{\nu(\mv)}\sum_{\me=(\mv,\mw)\in \partial\mV_0} \omega(\me),\qquad \mw\in \mV_0;
\end{equation}
we also write
\begin{equation}\label{eq:degdegout-def}
D_{\mV_0,{\mathrm{out}}}(\mv):=\diag(\degout(\mw))_{\mw\in\mV_0}.
\end{equation}

	{(In particular, $\degout(\mv)=0$ if $\mv\in\mV_0$, as a vertex in $\mV$, is only adjacent to vertices that also lie in $\mV_0$.)}

{
We observe that $\dFormD$ is bounded from below; indeed, it is immediate that, for any $\mV_0 \subset \mV$, $\dFormD (f) \geq 0$ for all $f \in D(\dFormD)$. Under the assumption that $\dLapl$ and thus $\dLaplD$ has compact resolvent, there is thus a smallest eigenvalue}
\begin{equation}\label{eq:minmax-dir}
	\begin{split}
	\lambda_1^{\mathrm{D}}(\mV_0) :&=\min\left\{\frac{\DirForm{\mV_0}(g)}{\|g\|_{\ell^2_{\nu_0}(\mV_0)}^2}~|~g\in D(\DirForm{\mV_0})\setminus\{0\}\right\}\\
	& = \min\left\{\frac{\dForm(f)}{\|f\|_{\lSP(\mV)}^2}~|~f\in D(\dForm)\setminus\{0\},~\supp(f)\subset\mV_0\right\},
	\end{split}
\end{equation}
where the minimum is attained by the eigenvectors corresponding to the eigenvalue \(\lambda_1^{\mathrm{D}}(\mV_0)\). {Note that in fact $\lambda_1^{\mathrm{D}}(\mV_0) > 0$ as long as $\mV_0$ is a proper subset of $\mV$ (and $\mG$ is connected), as follows from a standard argument {(see e.g. \cite[Proposition~2.5]{BifMug25})}.

\subsubsection{Boundaryless energy}\label{sec:boundaryless-energy}
In the following, we will consider other type of energies. However, unlike in the Dirichlet case, there will be no natural extension operators that allow us to define closed forms of the Laplacians on subgraphs.  We will restrict ourselves to $D(Q) = h^1(\mV)$ for the other energies that we consider; also, when dealing with the boundaryless energy and with the Neumann energy (see \autoref{sec:neumann-energy}), we will always assume that $\nu(\mV) < \infty$.

The first of these other Laplacians is simply the Laplacian of the subgraph $\mG_0$ induced by $\mV_0$, considered as a graph in its own right and without reference to $\mG\setminus\mG_0$ or $\mV\setminus\mV_0$. This operator thus acts on $\ell^2_{\nu_0}(\mV_0)$ without seeing the ambient graph. At least formally, in terms of the block decomposition in~\eqref{eq:block-dlapl}, this amounts to studying
\begin{equation}\label{eq:dlapl-boun-schr}
\dLaplB={\mathcal L}_{\mV_0,\mV_0}-D_{\mV_0,\mathrm{out}}=\dLaplD-D_{\mV_0,\mathrm{out}},
\end{equation}
where $D_{\mV_0,\mathrm{out}}$ is defined in~\eqref{eq:degdegout-def}.

We will, however, study the operator via its associated quadratic form \(\BouForm{\mV_0}\), which is easily seen to be given by
	\[
    \begin{split}
    \dFormB(g)&=
    \|P_{\mE_0}\mathcal I^\top \natExt{\mV_0}
    g \|^2_{\ell^2_{\omega}(\mE)}\\
    &=    \sum_{\me=\{\mv,\mw\}\in\mE_0}\omega(\me)|g(\mv)-g(\mw)|^2
	\end{split}
    \]
on the maximal domain
	\[D(\dFormB)=\left\{g\in \ell^2_{\nu_0}(\mV_0)~|~\sum_{\me=\{\mv,\mw\}\in\mE_0}\omega(\me)|g(\mv)-g(\mw)|^2<\infty\right\}=h^1(\mV_0).\]

(Here  $P_{\mE_0}$ is the orthogonal projector of $\ell^2_{\omega}(\mE)$ onto $\ell^2_{\omega_0}(\mE_0)$.)

We then take $\dLaplB$ to be the operator on $\ell^2_{\nu_0}(\mV_0)$ associated with $\dFormB$. By the general theory of graph Laplacians \cite{HaeKelLen12}, $\dLaplB$ is self-adjoint. We will call $\dLaplB$ the \emph{boundaryless Laplacian} since it is the natural maximal Laplacian on $\mG_0$ which does not see the boundary of $\mG_0$ as a subgraph of $\mG$.

We call $\dFormB$ the \emph{boundaryless energy form} (with respect to $\mV_0$); and if $D(\dFormB)$ is compactly embedded into $\ell^2_\nu(\mV_0)$ and, hence, $\dLaplB$ has pure point spectrum, then we call
\begin{equation}\label{eq:lambda2-rayl-chung}
\llB(\mV_0)=\min\left\{\frac{\dFormB(g)}{\|g\|_{\ell^2_{\nu_0}(\mV_0)}^2}~|~g\in D(\dFormB)\setminus\{0\},~\sum_{\mv\in\mV_0}\nu(\mv)g(\mv)=0\right\}.
\end{equation}
the  \emph{boundaryless energy} (with respect to $\mV_0$).
In particular, the associated discrete Laplacian on finite connected subgraphs {induced by vertex subsets $\mV_0$ with at least two vertices} has discrete spectrum and \(\llB(\mV_0)\) is well-defined {and strictly positive}.
{We use the notation $\llB(\mV_0)$ because, under our standing assumption that $\nu(\mV) < \infty$ and hence $\nu(\mV_0)<\infty$,} this is the second (variational) eigenvalue of $\dLaplB$, the first eigenvalue being $0$, with the constant functions being the corresponding eigenvectors.

In order to define the spectral energies we need additional assumptions on the underlying graph.
    
\begin{assum}\label{assum:unif-can-com}
There exists $C>0$ such that for all $\mV_0\subset \mV$ that induce a connected subgraph, and all $f\in D(\dFormB{\mV_0})$ with mean value zero on $\mV_0$ (that is, for which $\sum_{\mv \in \mV_0} \nu(\mv)f(\mv) = 0$), there holds
\begin{equation}\label{eq:def-uniform-can-comp}
	\|f\|_{\ell^\infty(\mV_0)}\leq C  (\BouForm{\mV_0}(f))^\frac12  .
\end{equation}

\end{assum}

{This is a kind of Poincar\'e inequality, which we require to hold uniformly in $\mV_0 \subset \mV$. We will discuss sufficient conditions for this property to hold for a given graph $\mG$ in \autoref{subsec:sufficientunif}.}

For $\mV_0 \subset \mV$, $\mV_0 \neq \emptyset$ we then define
\begin{equation}\label{eq:llb}
		\llB(\mV_0)=\min\left\{\frac{\BouForm{\mV_0}(g)}{\|g\|_{\ell^2_{\nu_0}(\mV_0)}^2}~|~g\in D(\BouForm{\mV_0})\setminus\{0\},~\sum_{\mv\in\mV_0}\nu(\mv)g(\mv)=0\right\}
\end{equation}
if each connected component of the subgraph induced by \(\mV_0\) has at least two elements. As in the Neumann case, if $\mV_0$ consists of a single element, then $\llB(\mV_0) = \infty$. Note that the subgraph induced by \(\mV_0\) is connected if and only if \(\llB(\mV_0)\) is positive, and, in that case, \(\llB(\mV_0)\) is the first positive eigenvalue of the discrete Laplacian on \(\ell^2_{\nu_0}(\mV_0)\). We refer to \(\llB(\mV_0)\) as the {boundaryless} energy of \(\mV_0\).

\subsubsection{Neumann energy}\label{sec:neumann-energy}

Under the same assumptions as for the boundaryless energy, namely $\nu(\mV) < \infty$ and $D(Q) = h^1(\mV)$,
given $\mV_0\subsetneq \mV$ and $f:\mV\to \C$, we refer to
\begin{equation}\label{eq:neum-bc}
\frac{1}{\nu(\mv)}\sum_{\me=\{\mv,\mw\}\in \partial\mV_0} \omega(\me)(f(\mv)-f(\mw))=0\qquad \hbox{for all }\mv\in \delta \mV_0
\end{equation}
as  \emph{Neumann conditions at the vertex boundary} $\delta \mV_0$ of $\mV_0$.

Following \cite[Section~2]{ShiYu25}, let $\neuExt{\mV_0}$ be the \textit{Neumann extension operator}, which maps functions supported on $\mV_0$ to the whole set $\mV$ as follows (we recall the vertex decomposition was introduced in \eqref{eq:block-dlapl}):
\begin{displaymath}
    \neuExt{\mV_0}f(\mv) = 
    \begin{cases}
        f(\mv) \qquad &\text{if } \mv \in \mV_0,\\
        \frac{\sum\limits_{\mw:\me=\{\mv,\mw\}\in\partial\mV_0}\omega(\me)f(\mw)}{\sum\limits_{\mw:\me=\{\mv,\mw\}\in\partial\mV_0}\omega(\me)} \qquad &\text{if } \mv \in \delta\mV_0,\\
        0 \qquad &\text{if } \mv \in \widetilde{\mV_0}
   \end{cases}
\end{displaymath}
(so that, in particular, the Neumann conditions \eqref{eq:neum-bc} hold at every vertex in $\delta\mV_0$).

We consider the quadratic form given by
\[
\dFormN(g):=\|P_{\overline{\mE_0}}\mathcal I^\top \neuExt{\mV_0}
    g \|^2_{\ell^2_{\omega}(\mE)} = \sum_{\me = \{ \mv, \mw\}\in \overline {\mE_0}} \omega(\me) \left | \neuExt{\mV_0} g(\mv) - \neuExt{\mV_0} g(\mw)\right |^2 
\]
on the maximal domain
	\[D(\dFormN):=\{g\in\ell^2_{\nu_0}(\mV_0)~|~\dFormN(g) < \infty\},
	\]
where $\overline{\mE_0}:=\mE_0\cup \partial \mV_0$ is the set of all edges in the subgraph induced by $\mV_0$ along with all edges that link one vertex in $\mV_0$ with one vertex outside $\mV_0$; {and $P_{\overline{\mE_0}}$ is the orthogonal projector of $\ell^2_{\omega}(\mE)$ onto $\ell^2_{\omega_0}(\overline{\mE_0})$}. 

\begin{lemma}
The quadratic form    $(\dFormN,D(\dFormN))$ is closed.
\end{lemma}
\begin{proof}
    One easily verifies that $\NeuForm{\mV_0}$ is a nonnegative, bilinear form. Let us now consider a Cauchy sequence $g_n$ with respect to the form norm $\|\cdot\|_{\ell^2_{\nu_0}}^2+\dFormN(\cdot)$. Indeed, then $g_n$ is a Cauchy sequence with respect to $\ell^2_{\nu_0}(\mV_0)$ and $\mathcal I^\top \neuExt{\mV_0} g$ is a Cauchy sequence with respect to $\ell^2_\omega(\mE)$. Then there exists $g\in \ell^2_{\nu_0}(\mV_0)$ such that $g_n \to g$ in $\ell^2_{\nu_0}(\mV_0)$ and, since the convergence is in particular pointwise, we conclude that $\mathcal I^\top \neuExt{\mV_0} g_n \to \mathcal I^\top \neuExt{\mV_0} g$ in $\ell^2_{\omega}(\mE)$. In particular, $g\in D(Q_{\mV_0})$ and $g_n \to g$ with respect to the form norm. 
    
\end{proof}

We refer to the associated self-adjoint operator as the Laplacian with Neumann boundary conditions at $\delta\mV_0$; this is the (possibly unbounded) operator on $\ell^2_\nu(\mV_0)$ formally given by
\begin{equation}\label{eq:neum-lapl-schr-new}
\begin{cases}
    &\dLaplN g(v) = \frac{1}{\nu(v)} \sum_{\me=\{\mv, \mw\}} \omega(\me) (g(\mv) - g(\mw))  \\
    &\frac{1}{\nu(\mv)}\sum_{\me=\{\mv,\mw\}\in \partial\mV_0} \omega(\me)(f(\mv)-f(\mw))=0\qquad \hbox{for all }\mv\in \delta \mV_0
\end{cases}
\end{equation}

Note that the argument in the proof of \autoref{lem:dlapld-selfadj} cannot be mimicked here, since $\neuExt{\mV_0}$ generally fails to be an isometry.  If $D(\dFormN)$ is compactly embedded into $\ell^2_\nu(\mV_0)$ and, hence, the associated operator $\dLaplN$ on $\ell^2_\nu(\mV_0)$ has pure point spectrum, then we call
\begin{equation}\label{eq:lambda2-rayl-chung-2}
\llN(\mV_0):=\min\left\{\frac{\dFormN(g)}{\|g\|_{\ell^2_{\nu_0}(\mV_0)}^2}~|~g\in D(\dFormN)\setminus\{0\},~\sum_{\mv\in\mV_0}\nu(\mv)g(\mv)=0\right\}.
\end{equation}
the  \emph{Neumann energy} (with respect to $\mV_0$); it is immediate that $\llN(\mV_0)=0$ if and only if the subgraph induced by $\mV_0$ is disconnected. 
As in the boundaryless case, the notation $\llN(\mV_0)$ reflects that, at least if $\dLaplN$ has a discrete spectrum, {zero will be the smallest eigenvalue of $\dLaplN$ will be zero (since $\nu(\mV) < \infty$ and thus the constant vectors will be in the null space of $\dLaplN$), and $\llN$ will be the second.} Similarly, as presented in \cite[Section~8.2]{Chu97} the Neumann conditions \eqref{eq:neum-bc} arise naturally via an alternative variational characterization of $\llN(\mV_0)$, but we do not discuss the approach here further.

{The question arises as to when the minimum in \eqref{eq:lambda2-rayl-chung-2} actually exists.} If $\mV_0$ consists of a single element, then no vectors satisfy the orthogonality condition, the set over which we are minimizing is empty, and so $\llN(\mV_0) = \infty$. For all sets $\mV_0$ with at least two elements, {it is enough to know that $D(\dFormN)$ embeds compactly into $\ell^2_{\nu_0}(\mV_0)$:} the existence of such a minimum is guaranteed by the following assumption, which will take on a role analogous to the canonical compactifiability condition of \autoref{def:canonically-compactifiable}, as well as to \autoref{assum:unif-can-com} for the boundaryless energy. We will impose this Poincar\'e-type inequality assumption throughout, whenever we consider Neumann energies.

\begin{assum}\label{assum:neum-exist}
There exists $C>0$ such that for all $\mV_0\subset \mV$ that induce a connected subgraph, and all $f\in D(\dFormN)$ with mean value zero on $\mV_0$ (that is, for which $\sum_{\mv \in \mV_0} \nu(\mv)f(\mv) = 0$), there holds
\[
\|f\|_{\ell^\infty(\mV_0)}\le C (\NeuForm{\mV_0}(f) )^\frac{1}{2}.
\]
\end{assum}

As with \autoref{assum:unif-can-com}, we will discuss sufficient conditions for this property in \autoref{subsec:sufficientunif}.

\begin{rem}\label{rem:stronguniformcanN}
	Under \autoref{assum:neum-exist} every minimizer of \eqref{eq:lambda2-rayl-chung-2} satisfies
	\[
	\| f\|_{\ell^\infty(\mV_0 \cup \delta \mV_0)} =\| f\|_{\ell^\infty(\mV_0)} \le C \NeuForm{\mV_0}(f)^\frac{1}{2},
	\]
    where the equality follows from~\eqref{eq:neum-bc}.

    Conversely, if the subgraph of $\mG$ with vertex set $\mV_0\cup \delta \mV_0$ and edge set $\mE_0\cup \partial \mV_0$ is canonically compactifiable, then also \autoref{assum:neum-exist} holds.
\end{rem}

We stress that \autoref{lem:dlapld-selfadj}.\eqref{item:monotonicity-dir} does not have a counterpart for $\llN(\mV_0)$.

\begin{rem}
\label{rem:dbn-inequalities}

It was proved in~\cite[Theorem~1.4 and Theorem 1.5]{ShiYu25} that on any \emph{finite} weighted graph $\lambda^{\mathrm{B}}_i(\mV_0)\le \lambda^{\mathrm{N}}_i(\mV_0)\le \lambda^{\mathrm{D}}_i(\mV_0)$, for all $i$ (where the higher eigenvalues can be defined by the natural min-max variational characterization which generalizes our definition of $\llD$, $\llB$ and $\llN$). Also, by \cite[Theorem~1.1]{ShiYu25}, if the first inequality is actually an equality, then there is a function on $\mG$ that minimizes both Neumann and boundaryless energy and that satisfies both Dirichlet and Neumann conditions at $\delta\mV_0$. In our setting, the inequality $\llB(\mV_0) \leq \llN(\mV_0)$, for any $\mV_0$ for which both quantities are well defined, is an immediate consequence of the respective variational characterizations \eqref{eq:llb} and \eqref{eq:lambda2-rayl-chung-2}, since $D(\dFormN) \subset D(\dFormB)$, the orthogonality condition is the same in both cases, and $\dFormN(g) \geq \dFormB(g)$ for all $g \in D(\dFormN)$. If one defines the higher eigenvalues in our setting, the same variational arguments will yield that the natural inequalities $\lambda^{\mathrm{B}}_i(\mV_0)\le \lambda^{\mathrm{N}}_i(\mV_0)\le \lambda^{\mathrm{D}}_i(\mV_0)$ will continue to hold under appropriate assumptions on $\mV_0$; however, we will not need these more general eigenvalues here and so do not go into details.
\end{rem}

\begin{exa}
\label{exa:prismatic}
We give a simple concrete example comparing the three different Laplacian energies discussed above.
Let $\mG$ be the prism graph in \autoref{fig:prisma}, {we assume that $\nu(\mv) = 1$ for all vertices $\mv$, and $\omega(\me) = 1$ for all edges $\me$.}
	\begin{figure}[H]
	\begin{tikzpicture}
\coordinate (v1) at (0,0);
\coordinate (v2) at (3,0);
\coordinate (v3) at (1.5,1.68);
\coordinate (w1) at (0,1.9);
\coordinate (w2) at (3,1.9);
\coordinate (w3) at (1.5,3.58);
\coordinate (z1) at (0,3.8);
\coordinate (z2) at (3,3.8);
\coordinate (z3) at (1.5,5.48);
\draw[fill] (v1) circle (1.95pt) node[anchor=east] {$\mv_1$};
\draw[fill] (v2) circle (1.95pt) node[anchor=west] {$\mv_2$};
\draw[fill] (v3) circle (1.95pt) node[anchor=east] {$\mv_3$};
\draw[fill] (w1) circle (1.95pt) node[anchor=east] {$\mw_1$};
\draw[fill] (w2) circle (1.95pt) node[anchor=west] {$\mw_2$};
\draw[fill] (w3) circle (1.95pt) node[anchor=east] {$\mw_3$};
\draw[fill] (z1) circle (1.95pt) node[anchor=east] {$\mz_1$};
\draw[fill] (z2) circle (1.95pt) node[anchor=west] {$\mz_2$};
\draw[fill] (z3) circle (1.95pt) node[anchor=east] {$\mz_3$};
\draw (v1) -- (w1) -- (z1);
\draw (v2) -- (w2) -- (z2);
\draw (v3) -- (w3) -- (z3);
\draw (v1) -- (v2) -- (v3) -- cycle;
\draw (w1) -- (w2) -- (w3) -- cycle;
\draw (z1) -- (z2) -- (z3) -- cycle;
\end{tikzpicture}
\caption{A prism graph}\label{fig:prisma}
	\end{figure}
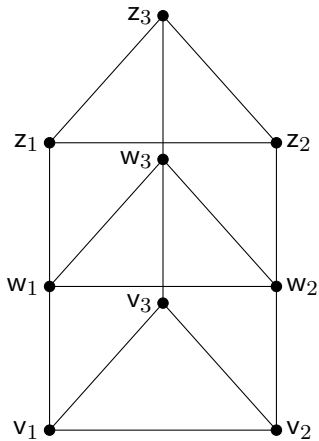
Consider the subgraphs induced by three different choices of $\mV_0$, namely the subgraph $\mG'_0$ and $\mG''_0$ induced by $\mV'_0=\{\mv_1,\mv_2,\mv_3\}$, $\mV''_0=\{\mv_1,\mv_2,\mw_2\}$, $\mV'''_0=\{\mv_1,\mw_1,\mz_1\}$, respectively.	
	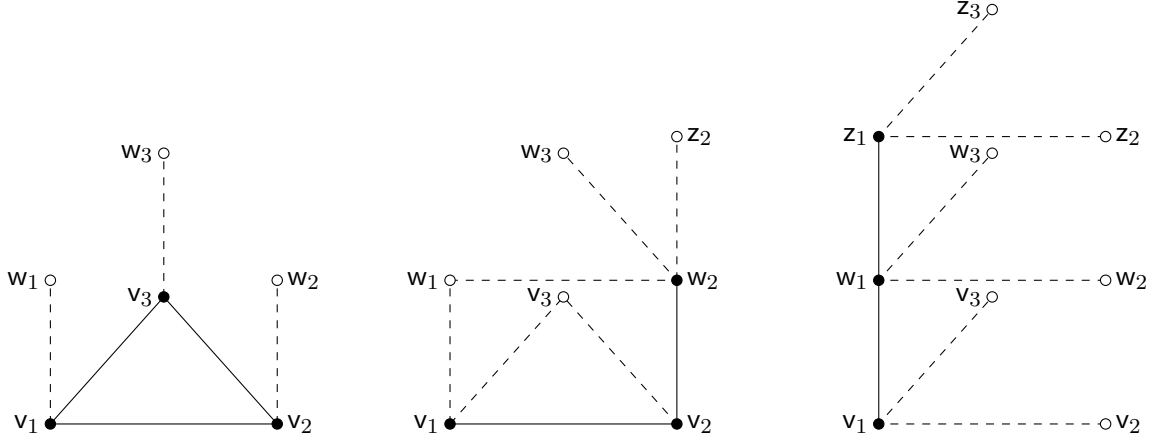
\begin{figure}[H]
	\begin{tikzpicture}
\coordinate (v1) at (0,0);
\coordinate (v2) at (3,0);
\coordinate (v3) at (1.5,1.68);
\coordinate (w1) at (0,1.9);
\coordinate (w2) at (3,1.9);
\coordinate (w3) at (1.5,3.58);
\draw[dashed] (v1) -- (w1);
\draw[dashed] (v2) -- (w2);
\draw[dashed] (v3) -- (w3);
\draw (v1) -- (v2) -- (v3) -- cycle;
\draw[fill] (v1) circle (1.95pt) node[anchor=east] {$\mv_1$};
\draw[fill] (v2) circle (1.95pt) node[anchor=west] {$\mv_2$};
\draw[fill] (v3) circle (1.95pt) node[anchor=east] {$\mv_3$};
\draw[fill=white] (w1) circle (1.95pt) node[anchor=east] {$\mw_1$};
\draw[fill=white] (w2) circle (1.95pt) node[anchor=west] {$\mw_2$};
\draw[fill=white] (w3) circle (1.95pt) node[anchor=east] {$\mw_3$};
\end{tikzpicture}
\qquad
	\begin{tikzpicture}
\coordinate (v1) at (0,0);
\coordinate (v2) at (3,0);
\coordinate (v3) at (1.5,1.68);
\coordinate (w1) at (0,1.9);
\coordinate (w2) at (3,1.9);
\coordinate (w3) at (1.5,3.58);
\coordinate (z2) at (3,3.8);
\draw (v1) -- (v2) -- (w2);
\draw[dashed] (v1) -- (w1) -- (w2) -- (w3);
\draw[dashed] (v1) -- (v3) -- (v2);
\draw[dashed] (w2) -- (z2);
\draw[fill] (v1) circle (1.95pt) node[anchor=east] {$\mv_1$};
\draw[fill] (v2) circle (1.95pt) node[anchor=west] {$\mv_2$};
\draw[fill=white] (v3) circle (1.95pt) node[anchor=east] {$\mv_3$};
\draw[fill=white] (w1) circle (1.95pt) node[anchor=east] {$\mw_1$};
\draw[fill] (w2) circle (1.95pt) node[anchor=west] {$\mw_2$};
\draw[fill=white] (z2) circle (1.95pt) node[anchor=west] {$\mz_2$};
\draw[fill=white] (w3) circle (1.95pt) node[anchor=east] {$\mw_3$};

\end{tikzpicture}
\quad
\qquad
	\begin{tikzpicture}
\coordinate (v1) at (0,0);
\coordinate (v2) at (3,0);
\coordinate (v3) at (1.5,1.68);
\coordinate (w1) at (0,1.9);
\coordinate (w2) at (3,1.9);
\coordinate (w3) at (1.5,3.58);
\coordinate (z1) at (0,3.8);
\coordinate (z2) at (3,3.8);
\coordinate (z3) at (1.5,5.48);
\draw (v1) -- (w1) -- (z1);
\draw[dashed] (v2) -- (v1) -- (v3);
\draw[dashed] (w2) -- (w1) -- (w3);
\draw[dashed] (z2) -- (z1) -- (z3);
\draw[fill] (v1) circle (1.95pt) node[anchor=east] {$\mv_1$};
\draw[fill=white] (v2) circle (1.95pt) node[anchor=west] {$\mv_2$};
\draw[fill=white] (v3) circle (1.95pt) node[anchor=east] {$\mv_3$};
\draw[fill] (w1) circle (1.95pt) node[anchor=east] {$\mw_1$};
\draw[fill=white] (w2) circle (1.95pt) node[anchor=west] {$\mw_2$};
\draw[fill=white] (w3) circle (1.95pt) node[anchor=east] {$\mw_3$};
\draw[fill] (z1) circle (1.95pt) node[anchor=east] {$\mz_1$};
\draw[fill=white] (z2) circle (1.95pt) node[anchor=west] {$\mz_2$};
\draw[fill=white] (z3) circle (1.95pt) node[anchor=east] {$\mz_3$};

\end{tikzpicture}
\caption{Three subgraphs $\mG'_0,\mG''_0,\mG'''_0$ of the prism graph $\mG$ in \autoref{fig:prisma} induced by $\mV'_0=\{\mv_1,\mv_2,\mv_3\}$, $\mV''_0=\{\mv_1,\mv_2,\mw_2\}$, and $\mV'''_0=\{\mv_1,\mw_1,\mz_1\}$, respectively. Here dashed lines denote edges in $\partial \mV_0$ and white dots represent vertices in $\delta \mV_0$.}\label{fig:prisma-parti}
	\end{figure}

For $\mG'_0$ we can check that
\[
\dLaplB=\dLaplN=
\begin{pmatrix}
2 & -1 & -1\\
-1 & 2 & -1\\
-1 & -1 & 2\\
\end{pmatrix}\qquad\hbox{and}\qquad
\dLaplD=
\begin{pmatrix}
3 & -1 & -1\\
-1 & 3 & -1\\
-1 & -1 & 3\\
\end{pmatrix}
\]
and therefore
\[
\llB(\mV'_0)=\llN(\mV'_0)=3,\qquad \llD(\mV'_0)=1.
\]
For $\mG''_0$,
\[
\dLaplB= 
\begin{pmatrix} 
1 & -1 &  0 \\ 
-1 &  2 & -1 \\ 
 0 & -1 &  1 
 \end{pmatrix},\qquad
\dLaplD=
\begin{pmatrix}
3 & -1 & 0\\
-1 & 3 & -1\\
0 & -1 & 4\\
\end{pmatrix},\qquad\hbox{and}\qquad
{\dLaplN=
\begin{pmatrix}
2 & -\frac32 & -\frac12\\
-\frac32 & \frac52 & -1\\
-\frac12 & -1 & \frac32\\
\end{pmatrix},
}
\]
so
\[
\llB(\mV''_0)=1,\qquad \llD(\mV''_0) {\approx 1.753},\qquad\hbox{and}\qquad { \llN(\mV''_0)=3-\frac{\sqrt{3}}{2} \approx 2.134},
\]
{where $\llD(\mV''_0)$ is, more precisely, the smallest root of $\lambda^3 - 10\lambda^2+31\lambda-29=0$.} Finally, for $\mG'''_0$,
\[
\dLaplB=\dLaplN=
\begin{pmatrix}
1 & -1 & 0\\
-1 & 2 & -1\\
0 & -1 & 1\\
\end{pmatrix}\qquad\hbox{and}\qquad
\dLaplD=
\begin{pmatrix}
3 & -1 & 0\\
-1 & 4 & -1\\
0 & -1 & 3\\
\end{pmatrix}
\]
and thus
\[
\llB(\mV'''_0)=\llN(\mV'''_0)=1,\qquad \llD(\mV'''_0)=2.
\]
In particular, this shows that different configurations may have smaller energy depending on the chosen realizations: in this example,
\begin{itemize}
\item $\llB(\mV'''_0)=\llB(\mV''_0)<\llB(\mV'_0)$,
\item $\llD(\mV'_0)<\llD(\mV''_0)<\llD(\mV'''_0)$,
\item $\llN(\mV'''_0)<\llN(\mV''_0)<\llN(\mV'_0)$.
\end{itemize}

\end{exa}

\subsection{Canonical compactifiability and beyond} \label{subsec:sufficientunif}

We first note that the three assumptions, canonical compactifiability (see \autoref{def:canonically-compactifiable}), \autoref{assum:unif-can-com} and \autoref{assum:neum-exist}, satisfy a natural hierarchy.

\begin{prop}
\label{prop:hierarchy}
    Let $\nu$ be finite. If $\mG$ satisfies \autoref{assum:unif-can-com}, then it satisfies \autoref{assum:neum-exist}. If it satisfies \autoref{assum:neum-exist}, then it is canonically compactifiable. Finally, if it is canonically compactifiable, then $h^1(\mV)$ embeds compactly in $\ell_\nu^2(\mV)$.
\end{prop}

\begin{proof}
    The first statement is an immediate consequence of the relations $D(\NeuForm{\mV_0})\subset D(\BouForm{\mV_0}) $ and $\BouForm{\mV_0}(f) \le \NeuForm{\mV_0}(f)$ for all $f\in D(\NeuForm{\mV_0})$.

    For the second, we note that \autoref{assum:neum-exist} can be rewritten in the form
    \begin{displaymath}
        \left |f(\mv) - \frac{1}{\nu(\mV_0)} \sum_{\mw \in \mV_0} \nu(\mw) f(\mw)\right | \leq C\BouForm{\mV_0}(f)^\frac12  ,
    \end{displaymath}
    which we then apply to $\mV_0 = \mV$  and to an arbitrary $f \in h^1$. Applying the reverse triangle inequality to the left-hand side of the above inequality and then Cauchy--Schwarz to the sum
    \begin{displaymath}
        \sum_{\mw \in \mV} \nu(\mw)|f(\mw)|
    \end{displaymath}
    yields the claim.

The final statement was noted in \autoref{cor:comp-embed-formdomain}.
\end{proof}

\begin{rem}
\label{rem:unif-can-com}
    \autoref{assum:unif-can-com} guarantees that the subgraph $\mG_0$ of $\mG$  induced by $\mV_0$ is itself canonically compactifiable, for every $\mV_0\subset \mV$. 
     We hence call such graphs satisfying the assumption \emph{uniformly compactifiable}, and will provide sufficient conditions for this property in \autoref{subsec:sufficientunif}.
\end{rem}

We next give geometric conditions that guarantee that \autoref{assum:unif-can-com}, and hence all three assumptions, are satisfied.

\begin{defi}

    A\emph{path} $\mP$ between two vertices $\mv, \mw\in \mV$ is an ordered finite sequence of vertices in the graph $\mG$ starting with $\mv$ and ending with $\mw$, such that each consecutive pair of vertices is connected by an edge in $\mG$. We will also assume that no vertex appears twice in the path (such objects are sometimes also called \emph{simple paths}). We will say that a path is
    \emph{non-contractible} 
    
    if, for any two vertices $\mv_i$ and $\mv_j$ in $\mP$, whenever $\mv_i$ and $\mv_j$ are adjacent in $\mG$, they are also adjacent in $\mP$.
     Otherwise, we say that the path is \emph{contractible}.

\end{defi}

(See also \autoref{fig:contractible} below.)

\begin{rem}\label{rem:contractible}
Suppose $\mv,\mw\in V$ are connected by a path $\mv_1, \mv_2, \ldots, \mv_N$, then we can always construct a noncontractible path connecting $\mv,\mw$. Indeed, suppose the induced subgraph has a cycle $\mv_k, \ldots, \mv_{m+k}, \mv_k$ of length $m$, then we can replace the subpath $\mv_k, \ldots, \mv_{m+k}$ by the edge $\mv_k, \mv_{m+k}$, which we refer to as a \emph{contraction} of the path. We may then continue to contract the resulting path and since the number of the cycles in the  subgraph induced by the path $\mv_1, \mv_2, \ldots, \mv_N$ was finite, this procedure terminates after a finite number of steps and we obtain a noncontractible path connecting $\mv, \mw$.
\end{rem}

\begin{figure}[H]
\centering
\includegraphics{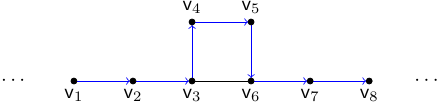}\hspace{1em} \includegraphics{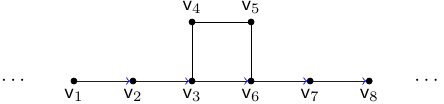}
\caption{Visualization of contractibility of paths. The path on the left can be contracted to the path on the right for example.}
\label{fig:contractible}
\end{figure}

If we try to extend \autoref{lem:finfingeohaekel} to arbitrary subgraphs, we run into problems since the diameter of an induced subgraph $\mG_0$ can be larger than that of $\mG$. We therefore introduce the following.

\begin{defi}
    We say that a graph has \textit{finite maximal path length} if 
    \begin{equation}
        L_{\text{max}}(\mG) := \sup_{\mv,\mw\in \mV}  \sup \sum_{i=1}^{n} \frac{1}{\omega(\mv_i, \mv_{i-1})} < \infty. 
    \end{equation}
    where the latter supremum is taken over all non-contractible, {finite} paths $\gamma=(\mv=\mv_0, \mv_1, \ldots, \mv_n=\mw)$ between $\mv,\mw$.
\end{defi}

\begin{prop}\label{prop:geometric-condition-for-ucc}
A connected graph \(\mG\) is \emph{uniformly compactifiable} if it has finite maximal path length.
\end{prop}

\begin{proof}

Let \(\mV_0 \subset \mV\), and let $\mv,\mw\in \mV_0$. If the subgraph $\mG_0$ induced by $\mV_0$ is connected, we can mimic the telescopic sum argument  in the first part of the proof of \cite[Lemma~3.4]{GeoHaeKel15}  (see also the proof of~\cite[Theorem~2.2]{Car00}) and pick a non-contractible path \(\gamma=(\mv_0,\mv_1,\ldots,\mv_n)\) in \(\mV_0\) that connects \(\mv\) and \(\mw\) to find
	\begin{equation*}
	\begin{split}
|f(\mv)|\le |f(\mv)-f(\mw)|+|f(\mw)| & \le L(\Gamma)^\frac12 \left(\sum_{j=1}^n  \omega(\mv_j,\mv_{j-1}) |f(\mv_j)-f(\mv_{j-1})|^2 \right )^\frac12 +|f(\mw)| \\
		& \leq L_{\max{}}(\mG)^\frac12 \BouForm{\mV_0}(f)^\frac12+
		\inf_{\mv\in\mV_0} \frac{1}{\nu(\mv)}\|f\|^2_{\ell^2_{\nu_0}(\mV_0)}.
	\end{split}
	\end{equation*}
    Since $\mw\in \mV_0$ can be chosen arbitrarily, we can deduce
    \begin{equation*}
        \left |f(\mv) - \frac{1}{\nu(\mV_0)} \sum_{\mw \in \mV_0} \nu(\mw) f(\mw)\right | \le \frac{1}{\nu(\mV_0)} \sum_{\mw\in \mV_0} \nu(\mw) |f(\mv) - f(\mw) | \le L_{\max}(\mG)^\frac12   \BouForm{\mV_0}(f)^\frac12  .
    \end{equation*}
    Since $\mv$ was arbitrary, this completes the proof.
    
The case of non-connected $\mG_0$ follows likewise, applying the above reasoning to each connected component.
    \end{proof}

\begin{cor}\label{cor:geometric-condition-for-ucc}
	A connected graph \(\mG\) satisfies \autoref{assum:neum-exist} and \autoref{assum:unif-can-com} if one of the following holds:
	\begin{enumerate}
	\item\label{item:fin-leng} \(\mG\) has finite total length, i.e.,
	\begin{equation}\label{eq:lengthdefinition}
		L(\mG):=\sum_{\me\in \mE} \frac{1}{\omega(\me)} <\infty,
	\end{equation}
	\item\label{item:fin-cycl} \(\mG\) has finite diameter \(\diam(\mG)<\infty\) and $\mG$ has only finitely many independent cycles.
	\end{enumerate}
\end{cor}

(A cycle is a non-contractible finite path $\gamma=(\mv_0,\ldots,\mv_n)$ with $\mv_0=\mv_n$. A set of independent cycles of $\mG$ is a basis of the \textit{finitary cycle space} in \cite[Chapter~8.7]{Die17}; roughly speaking, a cycle $\tilde{\mC}\subset \mG$ is dependent on $\mC_1,\mC_2$ if $\mC$ can be obtained by gluing $\mC_1,\mC_2$ together along their common edges and then erasing those common edges.)

\begin{proof} 
In case (1) we have $L_{\text{max}}(\mG) \le L(\mG) < \infty$ and we deduce that \autoref{assum:neum-exist} and \autoref{assum:unif-can-com} are satisfied by \autoref{prop:geometric-condition-for-ucc}.
	
Let us now consider case (2). Denote by $\mK$ a connected subgraph of $\mG$ that contains all of its independent cycles; by assumption $\mK$ is finite. We decompose \(\gamma\) into three successive paths \(\gamma_1,\gamma_2,\gamma_3\), so that \(\gamma_2\) is contained in \(\mathsf K\) and \(\gamma_1, \gamma_3\) are contained in one of the connected components of the forest $\mG\setminus \mK$. (Note that \(\gamma_1,\gamma_2,\gamma_3\) might be trivial paths consisting of only one vertex, if \(\gamma\) does not meet \(\mathsf K\) or components of the forest.) Since \(L(\gamma_1)+L(\gamma_3)\leq\diam(\mG)\) and \(L(\gamma_2)\leq L(\mathsf K)\), we infer $L_{\text{max}}\le \diam(\mG) + L(\mathsf K) < \infty$ and \autoref{assum:neum-exist} and \autoref{assum:unif-can-com} is satisfied by \autoref{prop:geometric-condition-for-ucc}.  
\end{proof}
\begin{rem}\label{rem:explicit-const}
	The proof of \autoref{prop:geometric-condition-for-ucc} shows that we may choose \(C=L(\mG)^\frac12  \) in \eqref{eq:def-uniform-can-comp}, if \(\mG\) has finite total length $L$ (as defined in \eqref{eq:lengthdefinition}), and \(C=\left(\diam(\mG)+L(\mathsf K)\right)^\frac12  \), if \(\mG\) has finite diameter and is a forest after removal of a finite vertex set \(\mathsf K\).
\end{rem}

\section{Spectral minimal partitions}\label{sec:smps}

In this section, we are going to define the notion of partition of a graph along with three different notions of energy associated with partitions. The main result of this article will be to prove that each such energy can be minimized over a reasonable class of graph partitions, under appropriate (and arguably fairly natural) conditions on the metric measure structure of the given graph. {The underlying graph $\mG = (\mV,\mE)$ will be fixed throughout; given a set $\mV_0 \subsetneq \mV$ of vertices, the induced subgraph $\mG_0 = (\mV_0,\mE_0)$ will include exactly those edges of $\mG$ which start and end at vertices in $\mV_0$: $\me =\{\mv,\mw\} \in \mE_0$ if and only if $\me \in \mE$ and $\mv,\mw \in \mV_0$.}

\begin{defi}[Partition]
\label{def:partition}
Let \(k\in\N\). A \emph{\(k\)-partition} \(\parti=\left(\mV_j\right)_{1\leq j\leq k}\) of \(\mG\) is a family of nonempty, pairwise disjoint subsets \(\mV_1,\ldots,\mV_k\subset \mV\) with
\begin{equation}\label{eq:VdisjuniVi}
\mV=\mV_1\cup\ldots\cup\mV_k,
\end{equation}
{such that each of the associated induced subgraphs $\mG_1,\ldots,\mG_k$ is connected.} We will refer to both the elements $\mV_j$ and their corresponding induced subgraphs as the \emph{clusters} of \(\mathcal P\). 
\end{defi}

\begin{rem}
We stress that, according to \autoref{def:partition}, any partition is necessarily \emph{exhaustive}, i.e., each vertex in $\mV$ must belong to $\mV_i$ for at least one (in fact, precisely one) $i$; beware of the difference with the more general setting allowed e.g.\ in~\cite{KenKurLen21} when dealing with \emph{metric} graphs.
\end{rem}

In particular, we will not generally distinguish between a set of vertices $\mV_j$ and the associated induced subgraph $\mG_j$ of $\mG$, since for every $\mV_j$ there is a unique, canonical associated $\mG_j$. Let \(\classpart(\mG)\) denote the set of \(k\)-partitions of \(\mG\), and let \(\classconn(\mG)\) denote the set of \(k\)-partitions whose clusters are connected. We also write $\mathfrak{P}(\mG) := \bigcup_{k \geq 1} \classpart (\mG)$ and $\mathfrak{C} (\mG) := \bigcup_{k\geq 1} \classconn (\mG)$.

\begin{defi}[Energy of partitions]
For a \(k\)-partition \(\parti=\left(\mV_j\right)_{1\leq j\leq k}\) we define the \emph{Dirichlet},  \emph{Neumann}, and \emph{boundaryless energy} of \(\parti\) by
\begin{align}
 \denergy[k,p](\parti) :&= \displaystyle\left( \frac{1}{k}\sum_{j=1}^k \lambda_1^{\mathrm{D}}(\mV_j)^p\right)^\frac{1}{p},\quad 1\le p<\infty,\qquad \denergy[k,\infty](\parti):= \max_{j=1,\ldots,k}\lambda_1^{\mathrm{D}}(\mV_j),\\
 \nenergy[k,p](\parti):&= \left( \frac{1}{k}\sum_{j=1}^k  \llN(\mV_j)^p\right)^\frac{1}{p},\quad 1\leq p<\infty,\qquad  \nenergy[k,\infty](\parti):= \max_{j=1,\ldots,k}\llN(\mV_j),\\
 \benergy[k,p](\parti):&= \left( \frac{1}{k}\sum_{j=1}^k  \llB(\mV_j)^p\right)^\frac{1}{p},\quad 1\leq p<\infty,\qquad  \benergy[k,\infty](\parti):= \max_{j=1,\ldots,k}\llB(\mV_j),
 \end{align}
respectively.

\end{defi}
We are in the following going to study these functionals and, in particular, whether minimizers  exist over the class of partitions whose clusters are connected: i.e.\ we are interested in the quantities
\begin{align*}
	\doptenergy[k,p](\mG) :&= \inf_{\mathcal P\in\classconn(\mG)}\denergy[k,p](\parti), & 	\noptenergy[k,p](\mG) :&= \inf_{\mathcal P\in\classconn(\mG)}\nenergy[k,p](\parti), & \boptenergy[k,p](\mG) :&= \inf_{\mathcal P\in\classconn(\mG)}\benergy[k,p](\parti),
\end{align*}
and whether there are respective partitions \(\mathcal P\in\classconn(\mG)\) with \(\doptenergy[k,p](\mG) = \denergy[k,p](\parti)\), \(\noptenergy[k,p](\mG) =\nenergy[k,p](\parti)\) or \(\boptenergy[k,p](\mG) =\benergy[k,p](\parti)\). Such partitions are called \emph{spectral minimal partitions} and the corresponding values \emph{spectral minimal energies}.

\subsection{Discrete convergence in \(\classpart\)} To study convergence of sequences of partitions, which means convergence of the corresponding subgraphs, we will work with the following notion of subgraph convergence, which may be understood as a kind of pointwise convergence.

\begin{defi}\label{def:set-convergence}
	A sequence \((\mV^n)_{n\in\N}\) of subsets of \(\mV\) is said to be \emph{(discretely) convergent} if for every \(\mv\in\mV\) there exists \(n_0(\mv)\in\N\) so that either
		\(\mv\in\mV^n\)
	for all \(n\geq n_0(\mv)\), or
		\(\mv\notin\mV^n\)
	for all \(n\geq n_0(\mv)\). In that case, we call
		\[\lim_{n\rightarrow\infty}\mV^n:=\{\mv\in\mV~|~\text{there exists $n_0 \in \N$ such that}~\mv\in\mV^n\text{ for all }n\geq n_0(\mv)\}\]
	the \emph{limit} of \((\mV^n)_{n\in\N}\).
\end{defi}

\begin{rem}
Observe that the notion of convergence in \autoref{def:set-convergence} is purely set-theoretic. Indeed, $(\mV^n)_{n\in \N}$ is convergent if and only if the corresponding sequence of characteristic functions $(\mathbf{1}_{\mV^n})_{n\in \N}$ is pointwise convergent.

{
The topology on the power set of $\mV$ associated with this notion of convergence is, however, metrizable. To see this, impose an arbitrary enumeration on $\mV$ (finite or infinite), $\mV=\{\mv_1, \mv_2, \ldots\}$, then we can define a metric via (for example)
\begin{equation}\label{eq:metrik-jm}
    d(\mW_1, \mW_2) = \sum_k 2^{-k} |\boldsymbol 1_{\mW_1}(\mv_k) - \boldsymbol 1_{\mW_2}(\mv_k)|, \qquad \mW_1,\,\mW_2 \subset \mV,
\end{equation}
then one easily verifies that $\mW_n\stackrel{d}\to \mW$ if and only if $\lim_{n\to \infty} \mW_n = \mW$.}

One could alternatively consider the notion of convergence induced by the Hausdorff metric with respect to the discrete metric on the set $\mV$ (cf.~\cite[Section~7.3]{BurBurIva01}), which we will not do, as \autoref{def:set-convergence} is well suited to our purpose. It can be shown that convergence in \autoref{def:set-convergence} is weaker than convergence in the Hausdorff metric (indeed, strictly weaker if $\mV$ is infinite).

\end{rem}

\begin{lemma}\label{lem:ex-conv-subsequ}
	For every sequence \((\parti^n)_{n\in\N}\) of \(k\)-partitions \(\parti^n=(\mV_j^n)_{1\leq j\leq k}\in\classpart(\mG)\) there exists a subsequence \((\parti^{n_l})_{l\in\N}\), so that each \((\mV_j^{n_l})_{l\in\N}\) is convergent.

\end{lemma}
\begin{proof}
	Since \(\mV\) is countable we may write	\(\mV=\{\mv^m~|~m\in\N\}\). Using induction one can show that for each \(m\in\N\) there exist a sequence \((n_l^m)_{l\in \N}\) in \(\N\) and an index \(j^m\in\{1,\ldots,k\}\) that satisfy the following properties for each \(m\in\N\):
	\begin{itemize}
		\item the sequence \((n_l^{m})_{l\in \N}\) is a subsequence of \((n_l^{m-1})_{l\in \N}\) if \(m\geq 2\),
		\item the vertex \(\mv^m\) is in \(\mV_{j^m}^{n_l^m}\) for each \(l\in\N\). 
	\end{itemize}
	(Note that the respective induction steps make use of the fact that each \(\parti^n\) only decomposes \(\mV\) into finitely many subsets.) These two properties already imply that \(\mv^m\) is in \(\mV_{j^m}^{n_l^M}\) for all \(M\geq m\) and \(l\in \N\). Now, let \((n_l)_{l\in\N}\) be the diagonal sequence given by \(n_l=n_l^l\) for \(l\in\N\). By construction, we have
	\begin{align*}		
		\mv^m &\in \mV_{j^m}^{n_l}, & \mv^m &\notin \mV_{j}^{n_l},\quad  j\neq j^m
	\end{align*}
	for all \(l,m\in\N\) with \(l\geq m\). Thus, each \((\mV_j^{n_l})_{l\in\N}\) is convergent to some $\mV_j^\infty$ for \(j=1,\ldots,k\). Finally, $\mV = \mV_1^\infty \cup \ldots \cup \mV_k^\infty$ since the same is true for each $n$, and the definition of convergence ensures that each $\mv \in \mV$ belongs to some $\mV_j^{n_\ell}$ for all $\ell$ large enough, and hence to some $\mV_j^\infty$.
\end{proof}

\begin{lemma}\label{lem:limit-is-partition}
	Let \((\parti^n)_{n\in\N}\) be a sequence of \(k\)-partitions \(\parti^n=(\mV_j^n)_{1\leq j\leq k}\), so that each sequence \((\mV_j^n)_{n\in\N}\) is convergent with limit \(\mW_j=\lim_{n\rightarrow \infty}\mV_j^n\). If all \(\mW_j\) are nonempty, then \(\parti=(\mW_j)_{1\leq j\leq k}\) is a \(k\)-partition.
\end{lemma}
\begin{proof}
	Let \(\mv\in\mV\). Since each \(\mathcal P^n\) is a \(k\)-partition and each cluster sequence $(\mV^n_j)_{n\in \N}$ converges, $j\in \{1,\ldots,k\}$, we find \(n_0(\mv)\in\N\) and \(j(\mv)\in\{1,\ldots,k\}\), so that for all \(n\geq n_0(\mv)\)
	\begin{align*}		
		\mv \in \mV_{j(\mv)}^{n}\qquad\hbox{and}\qquad \mv &\notin \mV_{\ell}^{n}\quad  \hbox{for all }\ell\neq j(\mv).
	\end{align*}
By definition of the limits \(\mW_j\), we obtain \(\mv \in \mW_{j(\mv)}\) and \(\mv \notin \mW_{j}\) for \(j\neq j(\mv)\). Since \(\mv\) was arbitrary, \(\mV\) is the disjoint union of the nonempty sets \(\mW_j\) and, thus, \(\parti\) is a \(k\)-partition.
\end{proof}

\begin{lemma}\label{lem:classconn-closed} 
Suppose that for any two vertices \(\mv\neq\mw\) in \(\mV\) there is only a finite number of noncontractible paths in \(\mG\) connecting \(\mv\) and \(\mw\). Then \(\mathfrak C(\mG)\) is sequentially closed in \(\mathfrak P(\mG)\) with respect to discrete convergence defined in \autoref{def:set-convergence}, i.e., the discrete limit  \(\mW\in\mathfrak P(\mG)\) of a sequence \((\mW^n)_{n\in\N}\) in \(\mathfrak C(\mG)\) already satisfies \(\mW\in\mathfrak C(\mG)\).
\end{lemma}
\begin{proof}
	Let \(\mv\neq\mw\) be two arbitrary vertices in \(\mW\). Note that there exists at least one noncontractible path by \autoref{rem:contractible}. Since \((\mW^n)_{n\in\N}\) converges discretely to \(\mW\), there exists some \(n_0\in\N\) with \(\mv,\mw\in\mW^n\) for all \(n\geq n_0\). For every \(n\geq n_0\) we choose a noncontractible path \(\gamma^n\) that connects \(\mv\) and \(\mw\) in the graph induced by \(\mW^n\), which exists by \autoref{rem:contractible}. Since there are only finitely many noncontractible paths connecting \(\mv\) and \(\mw\) in \(\mG\), we find a subsequence \((\mW^{n_k})_{k\in\N}\) of \((\mW^n)_{n\geq n_0}\), so that \(\gamma^{n_k}=\gamma\) for all \(k\in\N\) and some fixed path \(\gamma\) that connects \(\mv\) and \(\mw\) in \(\mG\), hence in \(\mW\). As \((\mW^{n_k})_{k\in\N}\) converges to \(\mW\) the path \(\gamma\) is contained in \(\mW\). Consequently \(\mv\) and \(\mw\) can be connected in \(\mW\). This proves the claim.
\end{proof}

\subsection{The Dirichlet case} We can now prove existence of spectral minimal partitions for the various eigenvalue problems.

\begin{lemma}\label{lem:connected-not-restr}
Let $\mG$ be any graph for which {the form domain in \eqref{eq:domdirform}} embeds compactly in $\ell^2_\nu(\mV)$. For every \(k\)-partition \(\parti\in \classpart(\mG)\) there exists a \(k\)-partition \(\tilde\parti\in \classconn(\mG)\) with \(\denergy[k,p](\tilde\parti)\leq\denergy[k,p](\parti)\).
\end{lemma}

{Recall that, under this compactness assumption,} therefore, on all subgraphs the Laplacian with Dirichlet boundary conditions has pure point spectrum, too, by \autoref{lem:dlapld-selfadj}.\eqref{item:monotonicity-ppsp}.

\begin{proof}
Let \(\parti=(\mV_j)_{1\leq j\leq k}\in \classpart(\mG)\setminus  \classconn(\mG)\). After reordering the clusters we may assume that there is some \(l\in \{1,\ldots,k\}\) so that \(\mV_1,\ldots,\mV_l\) are disconnected and \(\mV_{l+1},\ldots,\mV_k\) are connected. \\

	We decompose \(\mV_l\) into its connected components with respect to the topology induced by the shortest path metric in~\eqref{eq:distfunction}, i.e.\
		\(\mV_l=\bigcup_{i\in I} \mW_i\)
	for a family of nonempty and pairwise disjoint subsets \(\mW_i\subset \mV_l\) where \(\mW_{i_1}\cup\mW_{i_2}\) is disconnected for any pair \(i_1\neq i_2\) in \(I\). Note that the index set \(I\) is countable because \(\mV\) is countable. Now, since the Laplacian on \(\mV_l\) with Dirichlet boundary conditions has discrete spectrum, we have
		\[\lambda_1^{\mathrm{D}}(\mV_l)=\min_{i\in I}\lambda_1^{\mathrm{D}}(\mW_i),\]
	so we may choose an index \(i_0\) so that \(\lambda_1^{\mathrm{D}}(\mV_l)=\lambda_1^{\mathrm{D}}(\mW_{i_0})\). As \(\mV\) is connected, for each \(i\in I\setminus\{i_0\}\) there exist a \(j_i\in\{1,\ldots,k\}\) with \(j_i\neq l\) and a connected component \(\mU\) of \(\mV_{j_i}\) so that \(\mU_i\cup\mW_{i}\) is connected. Note that \(\mV_{j_i}\) is already connected if \(j_i>l\) and, thus, \(\mU_i=\mV_{j_i}\) in that case. Finally, we define a new partition \(\parti'=(\mV'_j)_{1\leq j\leq k}\) of \(\mV\) by setting
		\[\mV'_{l}:=\mW_{i_0}\]
	and
		\[\mV'_j:=\mV_j\cup\bigcup_{\substack{i\in I,j=j_i}}\mW_i, \quad j\neq l.\]
	By construction, \(\parti'\) is, in fact, a \(k\)-partition of \(\mG\), and \(\mV'_j\) is connected for \(j=l,\ldots,k\). Moreover, \(\mV_j\subset \mV'_j\) for \(j\neq l\), so the monotonicity of \(\lambda_1^{\mathrm{D}}\) (see \autoref{lem:dlapld-selfadj}.\eqref{item:monotonicity-dir}) yields \(\lambda_1^{\mathrm{D}}(\mV'_j)\leq\lambda_1^{\mathrm{D}}(\mV_j)\). Together with \(\lambda_1^{\mathrm{D}}(\mV_l)=\lambda_1^{\mathrm{D}}(\mW_{i_0})=\lambda_1^{\mathrm{D}}(\mV_l)\) we obtain
		\[\denergy[k,p](\parti')\leq\denergy[k,p](\parti).\]
	Altogether, we have constructed a \(k\)-partition \(\parti'\) with \(\denergy[k,p](\parti')\leq\denergy[k,p](\parti)\) that has at most \(l-1\) disconnected clusters. So, repeating this construction at most \(l-1\) more times proves the claim.
\end{proof}

\begin{theo}\label{thm:dir-existence}
	Let $\mG$ be any graph for which $h^1(\mV)$ embeds compactly in $\ell^2_\nu(\mV)$. For \(p \in [1,\infty]\) and \(k\in\N\) there exists a \(k\)-partition \(\parti\in \classconn(\mG)\) with \[\doptenergy[k,p](\mG)=\denergy[k,p](\parti).\]
\end{theo}

\begin{proof}
By \autoref{lem:connected-not-restr}, it suffices to show existence of \(\parti\in \classpart(\mG)\) with
	\[\doptenergy[k,p](\mG)\geq\denergy[k,p](\parti).\]
Let \((\parti^n)_{n\in\N}\) be a sequence of \(k\)-partitions \(\parti^n=(\mV_j^n)_{1\leq j\leq k}\in\classconn(\mG)\) with
\[\denergy[k,p](\parti^n)\searrow\doptenergy[k,p](\mG), \quad n\rightarrow\infty.\]
For all \(j=1,\ldots,k\) and \(n\in\N\) there exists some \(f_j^n\in \lSP(\mV)\) with \(\supp f_j^n\subset \mV_j^n\), \(||f_j^n||_{\lSP(\mV)}=1\) and \(\dForm(f_j^n)=\lambda_1^{\mathrm{D}}(\mV_j^n)\), thus
	\[\denergy[k,p](\parti^n)=\begin{cases}
		\left(\displaystyle\frac{1}{k}\sum_{j=1}^{k}\dForm(f_j^n)^p\right)^\frac{1}{p}, & 1\leq p<\infty\\
		\displaystyle\max_{j=1,\ldots,k}\dForm(f_j^n), & p=\infty.
	\end{cases}\]
	Therefore, the convergence \(\denergy[k,p](\parti^n)\rightarrow\doptenergy[k,p](\mG)\) yields that \((f_j^n)_{n\in\N}\) is bounded in \((D(\dForm),||\cdot||_\dForm)\) for \(j=1,\ldots,k\). By \autoref{lem:ex-conv-subsequ} and compactness of the embedding $h^1(\mV) \hookrightarrow \ell^2_\nu(\mV)$ we may assume that
\begin{itemize}	
	 \item each \((\mV_j^n)_{n\in\N}\) converges to some \(\mW_j\) for \(j=1,\ldots,k\) with respect to the topology in~\autoref{def:set-convergence},
	 \item each \((f_j^n)_{n\in\N}\) converges to some \(f_j\in\lSP(\mV)\) with respect to the norm \(||\cdot||_{\lSP(\mV)}\).
\end{itemize}
The convergence \(f_j^n\rightarrow f_j\) with respect to \(||\cdot||_{\lSP(\mV)}\) implies pointwise convergence \(f_j^n\rightarrow f_j\), which in turn implies \(\supp f_j\subset\mW_j\). And, since
	\[||f_j||_{\lSP(\mV)}=\lim_{n\rightarrow \infty}||f_j^n||_{\lSP(\mV)}=1,\]
every \(\mW_j\) is nonempty, so \(\parti:=(\mW_j)_{1\leq j\leq k}\) is a \(k\)-partition by \autoref{lem:limit-is-partition}. Now, for all \(j=1,\ldots,k\) Fatou's Lemma yields
	\[\dForm(f_j)\leq\liminf_{n\rightarrow\infty}\dForm(f_j^n)<\infty;\]
and we conclude \(f_j\in D(\dForm)\) with \(\supp f_j\in \mW_j\). Thus,
	\[\lambda_1^{\mathrm{D}}(\mW_j)\leq\dForm(f_j)\leq\liminf_{n\rightarrow\infty}\dForm(f_j^n)=\liminf_{n\rightarrow\infty}\lambda_1^{\mathrm{D}}(\mV_j^n)\]
and
	\[\denergy[k,p](\parti)\leq\liminf_{n\rightarrow \infty}\denergy[k,p](\parti^n) = \doptenergy[k,p](\mG).\qedhere\]
\end{proof}

\subsection{The Neumann case}

The main aim of this section is to prove a counterpart of \autoref{thm:dir-existence} for Neumann partitions, namely the following.

\begin{prop}\label{thm:neu-existence}

    Suppose that \(\mG\) satisfies \autoref{assum:neum-exist} and let $p \in [1,\infty]$ and $k\in \mathbb N$. Assume there exists a minimizing sequence $\mathcal P^n=(\mV_j^n)_{1\le j\le k}\in \classconn(\mG)$ for $\noptenergy[k,p]$ such that
\begin{equation}
\mV_j^n \to \mW_j, \qquad n\to \infty,
\end{equation}
for some subsets \(\mW_j\subset\mV\) for \(j=1,\ldots,k\). Then the following assertions hold.
\begin{enumerate}
\item Each \(\mW_j\) is non-empty. In particular, \(\parti:=(\mW_j)_{1\leq j\leq n}\in\classpart\).
\item If additionally \(\parti\in\classconn\), then \(\parti\) is a spectral minimal partition for \(\noptenergy[k,p](\mG)\).
\end{enumerate}
\end{prop}

\begin{proof}
To begin with, pick a sequence \( (\parti^n)_{n\in \N}\subset \classconn(\mG)\) of partitions \(\parti^n=(\mV^n_j)_{1\le j\le k}\) such that $\mV_j^n \to \mW_j$ for $j=1,\ldots, k$ and
\begin{equation}\label{eq:nen=nopt}
\lim_{n\to\infty}\nenergy[k,p](\parti^n)= \noptenergy[k,p](\mG).
\end{equation}

Because \autoref{assum:neum-exist} implies that $D(\NeuForm{\mV_j})$ is compactly embedded in $\ell^2_\nu(\mV_j)$, we deduce from the minimax characterisation of $\llN(\mV^n_j)$ that there exists $f^n_j\in D(\NeuForm{\mV^n_j})$ with $\|f^n_j\|_{\ell^2_\nu(\mV^{n}_j)}=1$, $\sum_{\mv\in\mV^n_j}f^n_j(\mv)\nu(\mv)=0$ and $\llN(\mV^n_j)=\NeuForm{\mV^n_j}(f^n_j)$.

Let $\overline{f^n_j}:=f^n_j\mathbf{1}_{\mV^n_j\cup \delta\mV^n_j}$: then with \autoref{rem:stronguniformcanN} we have
\[
\|\overline{f^n_j}\|_{\ell^\infty(\mV)}=\|{f^n_j}\|_{\ell^\infty(\mV^n_j)}\le C\NeuForm{\mV^n_j}(f^n_j)^\frac12=C\llN(\mV^n_j)^\frac{1}{2},
\]
where the sequence on the right-hand side is bounded because of~\eqref{eq:nen=nopt}. Also, because $\ell^\infty(\mV)$ is compactly embedded in $\ell^2_\nu(\mV)$, we may -- upon passing to a subsequence -- assume that  $\overline{f^n_j}$ converges in $\ell^2_\nu(\mV)$, say to $\overline{f_j}$, for all $j$. We conclude that on the one hand $\|\overline{f_j}\|_{\ell^2_\nu(\mV)}=\|\overline{f^n_j}\|_{\ell^2_\nu(\mV)}\equiv 1$, on the other hand $\supp(\overline{f_j})\subset \mW_j\cup \delta \mW_j$ (because $\overline{f_j}(\mv)=\lim_{n}\overline{f^n_j}(\mv)$), whence $\mW_j$ is guaranteed to be nonempty for all $j$.

Let us now introduce $f_j:=\overline{f_j}_{|\mW_j\cup{\delta \mW_j}}$: then
\[
\sum_{\mv\in\mV^n_j}f^n_j(\mv)\nu(\mv)=\sum_{\mv\in\mV^n_j}\overline{f^n_j}(\mv)\nu(\mv)=0,
\]
and passing to the limit in $n$ we find, as $\overline{f_j}=\lim_{n\to\infty} \overline{f^n_j}$ in $\ell^2_\nu(\mV)$, 
\[
\sum_{\mv\in\mW_j}f_j(\mv)\nu(\mv)=\sum_{\mv\in\mW_j}\overline{f_j}(\mv)\nu(\mv)=0.
\]
To conclude that $f_j$ is an admissible test function, it remains to observe that, by Fatou,
\[
\begin{split}
	\NeuForm{\mW_j}(f_j)&= \sum_{\me=\{\mv,\mw\}\in \mE} \mathbf 1_{(\mV_j\cup \delta \mV_j) \times (\mV_j\cup \delta \mV_j)}(\mv,\mw) \; \omega(\me) |\overline{f_j}(\mw)-\overline{f_j}(\mv)|^2\\ 
	&\leq  \liminf_{n\rightarrow \infty}  \sum_{\me=\{\mv,\mw\}\in \mE} \mathbf 1_{(\mV_j^n \cup \delta \mV_j^n) \times (\mV_j^n\cup \delta \mV_j^n)}(\mv,\mw) \; \omega(\me) |\overline{f_j^n}(\mw)-\overline{f_j^n}(\mv)|^2\\
	&= \liminf_{n\to \infty}\NeuForm{\mV^n_j}(f^n_j)<\infty
\end{split}
\]
i.e., $f_j\in D(\NeuForm{\mW_j})$ for all $j$.

It follows that
\[
\llN(\mW_j)\le \liminf_{n\to\infty}\llN(\mV^n_j)\qquad\hbox{for all }j,
\]
and finally
\[
\nenergy[k,p](\parti)\le \liminf_{n\to\infty}\nenergy[k,p](\parti^n)=\noptenergy[k,p](\mG).
\]
Because the converse inequality is trivial, the claims follows.
\end{proof}

We expect that if, additionally, $\parti\in\classconn(\mG)$, then there exists a spectral minimal partition attaining $\noptenergy[k,p](\mG)$. We can, at least, prove this in a relevant special case.
\begin{theo}\label{thm:existence-bdryless}
	Let \(k\in\N\) and \(p\in[1,\infty]\). Let $\mG$, in addition to \autoref{assum:neum-exist}, satisfy the requirements of \autoref{lem:classconn-closed}.
	
	Then there exists a $k$-partition $\mathcal P \in \classconn(\mG)$ with
    \[\noptenergy[k,p](\mG)=\nenergy[k,p](\parti).\]
\end{theo}
\begin{proof}
	Consider a sequence \((\parti^n)_{n\in\N}\) of \(k\)-partitions \(\parti^n=(\mV_j^n)_{1\leq j\leq k}\in\classconn(\mG)\) with
	\[\nenergy[k,p](\parti^n)\rightarrow\noptenergy[k,p](\mG),\quad \text{as }n\rightarrow\infty.\]
	By \autoref{lem:ex-conv-subsequ} we may assume that each \((\mV_j^n)_{n\in\N}\) is discretely convergent with some limit \(\mW_j\subset\mV\). As \((\parti^n)_{n\in\N}\) is a minimizing sequence, each \(\mW_j\) is non-empty by part \autoref{thm:neu-existence}.(1) and thus \(\parti:=(\mW_j)_{1\leq j\leq k}\in\classconn\) by \autoref{lem:classconn-closed}. Using \autoref{thm:neu-existence}.(2) we conclude \(\noptenergy[k,p](\mG)=\nenergy[k,p](\parti)\). This proves the claim.
\end{proof}

\subsection{The boundaryless case}
In this part we discuss the boundaryless case.

\begin{prop}\label{prop:minimizing-neumann-sequence}
Suppose that \(\mG\) satisfies \autoref{assum:unif-can-com} and let $p \in [1,\infty]$ and $k\in \mathbb N$. Assume there exists a minimizing sequence $\mathcal P^n=(\mV_j^n)_{1\le j\le k}\in \classconn(\mG)$ for $\boptenergy[k,p]$ such that
\begin{equation}
\mV_j^n \to \mW_j, \qquad n\to \infty,
\end{equation}
for some subsets \(\mW_j\subset\mV\) for \(j=1,\ldots,k\). Then the following assertions hold.
\begin{enumerate}
\item Each \(\mW_j\) is non-empty. In particular, \(\parti:=(\mW_j)_{1\leq j\leq n}\in\classpart\).
\item If additionally \(\parti\in\classconn\), then \(\parti\) is a spectral minimal partition for \(\boptenergy[k,p](\mG)\).
\end{enumerate}

\end{prop}
\begin{proof}
Let $(\mathcal P^n)_{n\in \mathbb N}$ be a minimizing sequence of $k$-partitions $\parti^n= (\mV_j^n)_{1\leq j \leq k} \in \classconn(\mG)$ with
\begin{equation}
\Lambda_{k,p}^B(\mathcal P^n) \to \mathcal L_{k,p}^B(\mG), \quad n\to \infty.
\end{equation}
In particular for $j\in \{1, \ldots, k\}$ and $n\in \mathbb N$ there exist $f_j^n \in D(\BouForm{\mV_j^n})$ with $\|f_j^n\|_{l^2(V_j)}=1$ and \(\sum_{\mv\in\mW_j}\nu(\mv)f_j^n(\mv)=0\) such that
\[\benergy[k,p](\parti^n)=\begin{cases}
		\left(\displaystyle\frac{1}{k}\sum_{j=1}^{k}\BouForm{\mV_j^n}(f_j^n)^p\right)^\frac{1}{p}, & 1\leq p<\infty\\
		\displaystyle\max_{j=1,\ldots,k}\BouForm{\mV_j^n}(f_j^n), & p=\infty.
	\end{cases}\]

Then we can extend the functions $f_j^n\in D(\BouForm{\mV_j^n})$ by zero to functions $\tilde f_j^n \in \lSP(\mV)$.  Then with \autoref{assum:unif-can-com} we infer
\begin{equation}
\|\tilde f_j^n\|_{\ell^\infty(\mv)} = \|f_j^n\|_{\ell^\infty(\mV_j^n)} \le C \|f_j^n\|_{\BouForm{\mV_j^n}}.
\end{equation}  
By \autoref{lem:ex-conv-subsequ} we may assume
that each \((\widetilde f_j^n)_{n\in\N}\) converges to some \(\widetilde f_j\in\lSP(\mV)\) with respect to the norm \(||\cdot||_{\lSP(\mV)}\). This implies pointwise convergence \(\tilde f_j^n\rightarrow \tilde f_j\), which in turn implies \(\supp \tilde f_j\subset\mW_j\). Let \(f_j:=\tilde f_{j|\mW_j}\in\ell^2_\nu(\mW_j)\). Since
	\[||f_j||_{\lSP(\mW_j)}=||\tilde f_j||_{\lSP(\mV)}=\lim_{n\rightarrow \infty}||\tilde f_j^n||_{\lSP(\mV)}=1,\]
every \(\mW_j\) is nonempty, so \(\parti:=(\mW_j)_{1\leq j\leq k}\) is a \(k\)-partition by \autoref{lem:limit-is-partition}. Now, we additionally assume that \(\parti\in\classconn\).  Fatou's Lemma yields

\[
\begin{split}
	\BouForm{W_j}(f_j)&= \sum_{\me=\{\mv,\mw\}\in \mE} \mathbf 1_{W_j \times W_j}(\mv,\mw) \; \omega(\me) |\tilde f_j(\mw)-\tilde f_j(\mv)|^2\\
	&\leq  \liminf_{n\rightarrow \infty}  \sum_{\me=\{\mv,\mw\}\in \mE} \mathbf 1_{\mV_j^n \times \mV_j^n}(\mv,\mw) \; \omega(\me) |\widetilde f_j^n(\mw)-\widetilde f_j^n(\mv)|^2\\
	&= \liminf_{n\rightarrow\infty} \BouForm{\mV_j^n}(f_j^n)<\infty
	\end{split}
	\]
for \(j=1,\ldots,k\). We conclude \(f_j\in D(\BouForm{W_j})\) and, as \((\tilde f_j^n)_{n\in\N}\) converges to \(\tilde f_j\) in \(\ell^2_\nu(\mV)\), we have \(\sum_{\mv\in\mW_j}\nu(\mv)f_j(\mv)=0\) for \(1\leq j\leq k\).  Therefore, \(f_j\) is a viable test function in the Courant--Fischer principle for \(\llB(\mW_j)\) and we obtain
	\[\llB(\mW_j)\leq\BouForm{\mW_j}(f_j)\leq\liminf_{n\rightarrow\infty}\BouForm{\mV_j^n}(f_j^n)=\liminf_{n\rightarrow\infty}\llB(\mV_j^n)\]
and
	\[\benergy[k,p](\parti)\leq\liminf_{n\rightarrow \infty}\benergy[k,p](\parti^n) = \boptenergy[k,p](\mG).\]
This proves the claim.
\end{proof}

\begin{theo}\label{thm:existence-neumann}
Let \(k\in\N\) and \(p\in[1,\infty]\). Let $\mG$, in addition to \autoref{assum:unif-can-com}, satisfy the requirements of \autoref{lem:classconn-closed}.

Then there exists a $k$-partition $\mathcal P \in \classconn(\mG)$ with
    \[\boptenergy[k,p](\mG)=\benergy[k,p](\parti).\]
\end{theo}
\begin{proof}
	Consider a sequence \((\parti^n)_{n\in\N}\) of \(k\)-partitions \(\parti^n=(\mV_j^n)_{1\leq j\leq k}\in\classconn(\mG)\) with
	\[\benergy[k,p](\parti^n)\rightarrow\boptenergy[k,p](\mG),\quad \text{as }n\rightarrow\infty.\]
	By \autoref{lem:ex-conv-subsequ} we may assume that each \((\mV_j^n)_{n\in\N}\) is discretely convergent with some limit \(\mW_j\subset\mV\). As \((\parti^n)_{n\in\N}\) is a minimizing sequence, each \(\mW_j\) is non-empty by part \autoref{prop:minimizing-neumann-sequence}.(1) and thus \(\parti:=(\mW_j)_{1\leq j\leq k}\in\classconn\) by \autoref{lem:classconn-closed}. Using \autoref{prop:minimizing-neumann-sequence}.(2) we conclude \(\boptenergy[k,p](\mG)=\benergy[k,p](\parti)\). This proves the claim.
\end{proof}

\begin{rem}
    We summarize what we know for a few special cases, based on the \autoref{thm:existence-bdryless}, \autoref{thm:existence-neumann} and \autoref{cor:geometric-condition-for-ucc}: there exists a spectral minimal partition for \(\noptenergy[k,p](\mG)\) and \(\boptenergy[k,p](\mG)\) if \(\mG\) satisfies either of the following two conditions:
	\begin{enumerate}
	\item \(\mG\) has finite total length \(L(\mG)<\infty\) and for any two vertices \(\mv\neq\mw\) there is only a finite number of non-contractible paths in \(\mG\) connecting \(\mv\) and \(\mw\);
	\item \(\mG\) has finite diameter \(\diam(\mG)<\infty\) and a finite number of independent cycles.
	\end{enumerate}
    In the latter case there exists a connected spectral minimal partition.
\end{rem}

\section{Spectral inequalities}\label{sec:spec-ineq}

In this section we wish to investigate the relationships among the minimal energies $\doptenergy[k,p](\mG)$, $\noptenergy[k,p](\mG)$, $\boptenergy[k,p] (\mG)$, and the relationship between these and geometric quantities of the whole graph $\mG$, as well as the Laplacian eigenvalues of the latter.

\begin{lemma}\label{lem:continuity-p}
    Let $k \in \N$, let $\nu$ be finite and suppose $\mG$ has finite total length. The functions
    \[
    p \mapsto \boptenergy[k,p] (\mG),\quad p \mapsto \doptenergy[k,p](\mG), \quad p \mapsto \noptenergy[k,p](\mG)
    \]
    are continuous and non-decreasing functions of $p \in [1,\infty]$.
\end{lemma}

\begin{proof}
    We will sketch the proof for $\doptenergy[k,p](\mG)$, the other two cases being analogue. It is easy to see that for $1\le p\le q < \infty$ we have
    \[
        \doptenergy[k,\infty](\mG) \le k^\frac{1}{q}\doptenergy[k,q](\mG) \le k^\frac{1}{p}\doptenergy[k,p](\mG) \le k^\frac{1}{p} \doptenergy[k,q](\mG) \le k^\frac{1}{p} \doptenergy[k,\infty](\mG)
    \]
    and the statement follows as an easy exercise.
\end{proof}

\begin{prop}\label{prop:comparison-bn}
    Let $\nu$ be finite, and let $\mG$ have finite total length. Then for any $k \in \N$ and any $1 \leq p \leq \infty$,  we have
    \[
    \boptenergy[k,p]\le \noptenergy[k,p].
    \]
\end{prop}

\begin{proof}
This is a direct consequence of the inequality $\llB(\mV_0)\le \llN(\mV_0)$ for any subset $\mV_0 \subset \mV$, see \autoref{rem:dbn-inequalities} (and compare \cite[Theorem~1.1]{ShiYu25}).
\end{proof}

\begin{prop}\label{prop:estimate-bopten}
Let $k\in \mathbb N$. Let $\nu$ be finite and assume $\mG$ to have finite total length. Then for any $1 \leq p \leq \infty$
\begin{equation}\label{eq:estim-bopten}
\boptenergy[k,p](\mG) \ge \frac{4k^2}{\nu(\mV) L(\mG)} .
\end{equation}
\end{prop}

\begin{proof}
Suppose first that $1 \leq p < \infty$. By \cite[Corollary~3.7]{LenSchSto18} (with obvious changes due to the fact that $\nu$ is finite but not necessarily a probability measure) we have that
\begin{equation}\label{eq:subopt-estim}
\llB(\mV_0)\ge \frac{4}{\nu(\mV_0)\diam(\mG_0)}\ge \frac{4}{\nu(\mV_0)L(\mG_0)}
\end{equation}
holds for any vertex set $\mV_0\subset \mV$ and the corresponding induced subgraph $\mG_0$,
see also~\cite[Theorem~4.2]{Moh91b} in the unweighted case. Accordingly, by Jensen's inequality (for $\varphi(x)=x^{-2p}$), for any $k$-partition $\parti$,
\[ 
\begin{split}
\frac{1}{4^p k}\sum_{j=1}^k[\llB(\mV_j)]^p&\ge \frac{1}{k} \sum_{j=1}^k\frac{1}{[\nu(\mV_j)L(\mG_j)]^p} = \frac{1}{k} \sum_{j=1}^k \phi(\nu(V_j)^\frac12   L(\mG_j)^\frac12  )\\
&\ge  \left ( \frac{1}{k} \sum_{j=1}^k \nu(\mV_j)^\frac12   L(\mG_j)^\frac12   \right )^{-2p}\\
&\ge \left ( \frac{1}{k}  \left ( \sum_{j=1}^k\nu(\mV_j)\right )^\frac12   \left ( \sum_{j=1}^kL(\mG_j)\right )^\frac12    \right )^{-2p} \ge \left(\frac{k^{2}}{\nu(\mV) L(\mG)}\right)^p,
\end{split}
\]
using the fact that any two clusters have disjoint vertex (and, hence, edge) sets.
Therefore, for each partition $\parti$ we have
\[
\benergy[k,1](\parti)\ge \frac{4k^2}{\nu(\mV)L(\mG)}
\]
and taking the $\inf$ of the left hand side yields
\begin{equation}
\boptenergy[k,1](\mG) \ge \frac{4k^2}{\nu(\mV) L(\mG)} ,
\end{equation}
as we wanted to prove. For $p=\infty$ just use the continuity statement in \autoref{lem:continuity-p}.
\end{proof}

\begin{prop}\label{prop:estimate-dopten}
Let $k\in \mathbb N$. Let $\nu$ be finite and assume $\mG$ to have finite total length. Then for $1 \leq p \leq \infty$
\begin{equation}\label{eq:estim-dopten}
\doptenergy[k,p](\mG) \ge \frac{k^2}{2\nu(\mV) L(\mG)} .
\end{equation}
\end{prop}
\begin{proof}
We replace~\eqref{eq:subopt-estim} by
\begin{equation}\label{eq:subopt-estim-dir}
\lambda_1^{\mathrm{D}}(\mV_j)\ge \frac{1}{\nu(\mV_j)\inr(\mG_j)}\ge \frac{1}{\nu(\mV_j)L(\mG_j)},
\end{equation}
see~\cite[Corollary~4.6]{LenSchSto18}, and carry the calculation through with $\inr(\mG_j)$ in place of $L(\mG_j)$, leading to
\begin{displaymath}
    \frac{1}{k}\sum_{j=1}^k[\llB(\mV_j)]^p \ge \left(\frac{k^{2}}{\nu(\mV) \sum_{j=1}^k \inr(\mG_j)}\right)^p.
\end{displaymath}
Take any $k$ paths $\mathcal{P}_j$ realizing the respective inradii $\inr (\mG_j)$ (so that $\inr (\mG_j) = L(\mathcal{P}_j)$). Note that the paths are not necessarily disjoint, since they may include edges that connected $\mG_j$ to its complement; however, each such edge can belong to at most two paths. Hence we have the inequality
\begin{displaymath}
    \sum_{j=1}^k \inr (\mG_j) \leq 2L(\mG),
\end{displaymath}
from which the claim now follows.
\end{proof}

\begin{prop}\label{prop:estimate-dopten-bimu}
Let $k\in \mathbb N$. Let $\nu\equiv 1$ and $\omega\equiv 1$ for a finite graph $\mG$ on $n$ vertices. Then for $1\leq p \leq \infty$
\begin{equation}\label{eq:estim-dopten-b} 
\doptenergy[k,p](\mG) \ge 2\left(1-\cos\frac{k\pi}{2n-1}\right)
\end{equation}
and
\begin{equation}\label{eq:estim-dopten-c} 
\boptenergy[k,p](\mG) \ge 2\left(1-\cos\frac{k\pi}{n}\right)
\end{equation}
\end{prop}

\begin{proof}
For $p=1$ the proof of this inequality can be carried out like \autoref{prop:estimate-bopten}, up to replacing the isoperimetric inequality \eqref{eq:subopt-estim}
by \cite[Corollary~5.12]{BifMug25} or \cite[Section 4.3]{Fie73} (for the Dirichlet case or the boundaryless case, respectively), invoking convexity of \(x\mapsto 1-\cos\frac{k\pi }{2x+1}\), and, in the Dirichlet case, observing that at most $n-1$ vertices can be outside the set of vertices on which Dirichlet conditions are imposed.

The statement then follows from the monotonicity statement in \autoref{lem:continuity-p}.
\end{proof}

\begin{prop}\label{prop:estimate-nopten}
Let $k\in \mathbb N$ and $1 \leq p \leq \infty$. Let $\nu$ be finite and assume $\mG$ to have finite total length. Then
\begin{equation}\label{eq:estim-nopten}
\noptenergy[k,p](\mG) \ge \frac{4k^2}{\nu(\mV) L(\mG)} .
\end{equation}
\end{prop}

\begin{proof}
This is an immediate consequence of \autoref{prop:estimate-bopten} and \autoref{prop:comparison-bn}.
\end{proof}
Further estimates could be found using the lower eigenvalue bounds in~\cite{ShiYu25}.

\begin{prop}
\label{prop:ld-ev}
    Let $\nu$ be finite and assume  the embedding $h^1(\mV) \hookrightarrow \ell_\nu^2(\mV)$ to be compact. Then for any $k \in \mathbb N$,
    \begin{displaymath}
        \lambda_k(\mG) \leq 2\doptenergy[k,\infty](\mG)
    \end{displaymath}
    where $\lambda_k$ is the $k$-th eigenvalue of the discrete Laplacian on $\mG$ (enumerated so that $\lambda_1(\mG) = 0$).
\end{prop}

Note that the inequality $\lambda_k \leq \doptenergy[k,\infty]$, true for the corresponding spectral minimal partition problem on Euclidean domains (and quantum graphs), is not true in this context, as the simple example of a graph with two vertices $\mv,\mw$ (and constant weights, $\nu(\mv)=\nu(\mw)=\omega(\{\mv,\mw\})=1$) shows: it is immediate that $\lambda_2=2$, while $\doptenergy[k,\infty]=1$.

\begin{proof}[Proof of \autoref{prop:ld-ev}]
    Denote by $\parti$ any $k$-partition of $\mG$, and by $\psi_1,\ldots,\psi_k$ the eigenvectors of the clusters, considered as elements of $\ell_\nu^2(\mV)$ after extension by zero, and normalized to have $\ell_\nu^2$-norm $1$. It is an immediate consequence of the Courant--Fischer min-max principle that
    \begin{displaymath}
        \lambda_k(\mG) \leq \max_{\substack{0 \neq u \in \aufspan \{\psi_1,\ldots,\psi_k\}\\ \|u\|_{\ell_\nu^2(\mV)}=1}} \sum_{\me = \{\mv,\mw\} \in \mE} \omega(\me)|u(\mv)-u(\mw)|^2.
    \end{displaymath}
    Take $u=\sum_{j=1}^k \alpha_j \psi_j$, where $\sum_{j=1}^k |\alpha_j|^2 = 1$. For each edge $\me$, we consider three cases for $\me$: (1) $\me$ is contained in some $\mG_j$; (2) $\me$ is on the boundary of some $\mG_j$; (3) $\me$ is contained in the complement of the $\mG_j$ together with their boundary layers. In case (3) we automatically have $u(\mv)=u(\mw)=0$; in case (1), $u(\mv)-u(\mw) = \alpha_j(\psi_j(\mv)-\psi_j(\mw))$.
    
    In case (2), if $\me$ is only on the boundary of $\mG_j$ and no other cluster, then the same conclusion is true as in case (1). Otherwise, the only other possibility is that $\me$ is a boundary edge between $\mG_j$ and another cluster $\mG_i$. In this case
    \begin{displaymath}
        |u(\mv)-u(\mw)|^2 = |\alpha_j\psi_j(\mv)-\alpha_i\psi_i(\mw)|^2
        \leq 2(|\alpha_j\psi_j(\mv)|^2 + |\alpha_i\psi_i(\mw)|^2).
    \end{displaymath}
    Summing over all edges $\me$ leads to
    \begin{multline*}
        \sum_{\me = \{\mv,\mw\} \in \mE} \omega(\me)|u(\mv)-u(\mw)|^2 \\
        \qquad 
        \leq \sum_{j=1}^k \left(\sum_{\me=\{\mv,\mw\} \subset \mG_j} \omega(\me)|\alpha_j|^2|\psi_j(\mv)-\psi_j(\mw)|^2 + \sum_{\mv \in \mV_j} 2\degout^{\mV_j}(\mv)|\alpha_j|^2|\psi_j(\mv)|^2\right)
    \end{multline*}
    (using the notation of \eqref{eq:degout-def}), which in turn leads to
    \begin{displaymath}
        \sum_{\me = \{\mv,\mw\} \in \mE} \omega(\me)|u(\mv)-u(\mw)|^2
        \leq 2\sum_{j=1}^k |\alpha_j|^2 \lambda_1^{\mathrm{D}}(\mG_j).
    \end{displaymath}
    Since $\sum_{j=1}^k |\alpha_j|^2 = 1$ and the $k$-partition $\parti$ was arbitrary, the claim now follows.
\end{proof}

\section{A family of examples}
\label{sec:snakes-and-ladders}

In this section we will study a model case of the spectral minimal partitions of the two-ladder graph $\mLn= \mP_n \times \mP_2$ (for short, ladders), where $\mP_n$ is a path graph with $n$ vertices, $2 \leq n \leq \infty$, see \autoref{fig:paths-to-the-ladder} and \autoref{fig:two-ladder-intro}. For simplicity we will take partitions into $k=2$ cells, and $p=\infty$, we will however compare all three energy functionals, namely Dirichlet, boundaryless and Neumann.

We will distinguish between the finite ($n<\infty$) and infinite ($n=\infty$) cases, with rather different goals in each case: 
\begin{itemize}
    \item In the finite case, where we can take all weights to be one and can thus determine minimizing partitions analytically, we will illustrate how the different functionals identify different kinds of minimizers {(we observed something similar for the eigenvalues themselves in \autoref{exa:prismatic})};
    \item In the infinite case, we will illustrate how without suitable summability conditions on the weights, minimizing partitions need not exist, and indeed the problem need not be well defined.
\end{itemize}

\begin{figure}[ht]
\centering

\begin{tikzpicture}[scale=1, line cap=round, line join=round]
  \def\n{5} 
  \foreach \i in {0,...,\n} {
    \fill (\i,0) circle (1.5pt);
  }
  \foreach \i in {0,...,\numexpr\n-1\relax} {
    \draw[thick] (\i,0) -- (\i+1,0);
  }
  \node at (\n/2,-0.55) {$\mP_n$};
\end{tikzpicture}
\hspace{2em} 
\begin{tikzpicture}[scale=1, line cap=round, line join=round]
  \coordinate (a) at (0,0);
  \coordinate (b) at (0,1.5);

  \draw[thick] (a) -- (b);
  \fill (a) circle (1.5pt);
  \fill (b) circle (1.5pt);
  \node[left] at (0,0.75) {$\mP_2$};
\end{tikzpicture}

\caption{The path graphs $\mP_n$ (for $n=6$) and $\mP_2$...}
\label{fig:paths-to-the-ladder}
\end{figure}
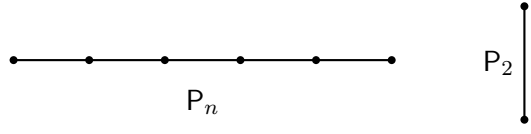

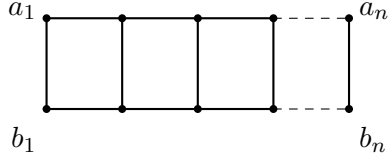
\begin{figure}[ht]
\centering
    \begin{tikzpicture}[scale=1, line cap=round, line join=round]
  \def\n{3} 
  \def\h{1.2}

  \foreach \i in {0,...,\n} {
    \fill (\i,0) circle (1.5pt);
    \fill (\i,\h) circle (1.5pt);
  }

  \foreach \i in {0,...,\numexpr\n-1\relax} {
    \draw[thick] (\i,0) -- (\i+1,0);
    \draw[thick] (\i,\h) -- (\i+1,\h);
  }

  \foreach \i in {0,...,\n} {
    \draw[thick] (\i,0) -- (\i,\h);
  }

  \draw[dashed] (\n, 0) -- (\n +1, 0);
  \draw[dashed] (\n, \h) -- (\n+1, \h);
  \draw[thick] (\n+1, 0) -- (\n+1, \h);

\fill (\n+1,0) circle (1.5pt);
\fill (\n+1, \h) circle (1.5pt);

  \node[left] at (0,-0.4) {$b_1$};
  \node[left] at (0,\h+0.1) {$a_1$};
  \node[right] at (\n+1,\h+0.1) {$a_{n}$};
  \node[right] at (\n+1,-0.4) {$b_{n}$};

\end{tikzpicture}

\caption{... and the two-ladder graph $\mLn=\mP_n \times \mP_2$.}
\label{fig:two-ladder-intro}
\end{figure}

\subsection{Finite ladders} We first consider the case of finite graphs, i.e. $n < \infty$, and assume for the meantime that all vertex and edge weights are $1$, i.e. 
\[
\nu(\mv) = \omega(\me) = 1 \qquad\hbox{for all }\mv \in \mV (\mLn)\hbox{ and all }\me \in \mE (L_n).
\]
Clearly, on any finite graph, including the $\mLn$, all our assumptions are satisfied; thus there will always exist minimizing partitions for all three of our functionals, for all $1 \leq k \leq 2n = |\mV(\mLn)|$. However, as noted, for simplicity we will only consider $k=2$; in this case, since our partitions need to be exhaustive, all clusters of all admissible $2$-partitions will be subgraphs of $\mL_n$ of the following form.

\begin{defi}
\label{def:ladder-clusters}
    Denote by $\mG_{\ell,m}\subset \mathbb N \times \mP_2$ the graph obtained by attaching a pendant path $\mP_m$, $m \geq 0$ to the ladder with $\ell \geq 0$ rungs, $\mP_\ell \times \mP_2$ (see \autoref{fig:ladder-clusters}; the case $m=0$ is taken to stand for a ladder of depth $\ell$, the case $\ell=0$ corresponds to a path of length $m$).
\end{defi}

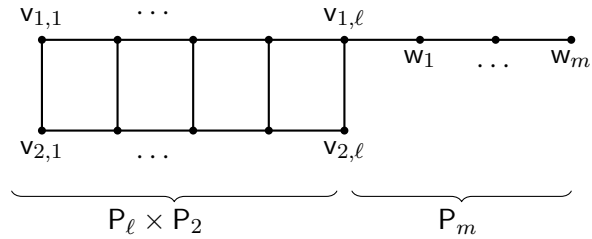
\begin{figure}[ht]
\begin{tikzpicture}[scale=1, line cap=round, line join=round]
  \def\h{1.2}
  \def\n{4}

  \foreach \i in {0,...,\n} {
    \fill (\i,0) circle (1.5pt);
    \fill (\i,\h) circle (1.5pt);
  }

  \foreach \i in {0,...,\numexpr\n-1\relax} {
    \draw[thick] (\i,0) -- (\i+1,0);
    \draw[thick] (\i,\h) -- (\i+1,\h);
  }

  \foreach \i in {0,...,\n} {
    \draw[thick] (\i,0) -- (\i,\h);
  }

  \draw[thick] (\n,\h) -- (\n+1,\h) -- (\n+2,\h) -- (\n+3,\h);

  \fill (\n+1,\h) circle (1.5pt);
  \fill (\n+2,\h) circle (1.5pt);
  \fill (\n+3,\h) circle (1.5pt);

  \node at (1.5,-0.35) {$\cdots$};
  \node at (1.5,\h+0.35) {$\cdots$};

  \node at (\n+2,\h-0.35) {$\cdots$};

  \node[below] at (0,0) {$\mv_{2,1}$};
  \node[above] at (0,\h) {$\mv_{1,1}$};
  \node[below] at (\n,0) {$\mv_{2,\ell}$};
  \node[above] at (\n,\h) {$\mv_{1,\ell}$};
  \node[below] at (\n+3,\h) {$\mw_{m}$};

  \node[below] at (\n+1,\h) {$\mw_{1}$};

  \draw[decorate,decoration={brace,amplitude=4pt,mirror}]
    (-0.4,-0.8) -- (\n-0.1,-0.8);
  \node at (1.5,-1.2) {$\mP_\ell \times \mP_2$};

  \draw[decorate,decoration={brace,amplitude=4pt,mirror}]
    (\n+0.1,-0.8) -- (\n+3,-0.8);
  \node at (5.5,-1.2) {$\mP_m$};
\end{tikzpicture}
\caption{The graph $\mG_{\ell,m}$ consisting of a ladder of length $\ell \geq 0$, $\mP_\ell \times \mP_2$, to which is appended a path graph $\mP_m$ on $m \geq 0$ vertices.}
\label{fig:ladder-clusters}
\end{figure}

Before identifying the minimizers, we will need a few preliminary results. The first is ostensibly for the boundaryless case but will yield the Neumann case as a corollary. Note that it is valid on more general graphs than just $\mL_n$.

\begin{lemma}\label{lem:optimal-config-boundaryless}
    Suppose that $\mG$ has $k\cdot n$ vertices and suppose all vertex and edge weights are $1$. If there exists a $k$-partition $\mathcal P=\{ \mV_1, \ldots, \mV_k\}$ such that for each $i=1,\ldots,k$ the induced subgraph $\mG_i$ associated with $\mV_i$ is a path of length $n$, then $\mathcal P$ is minimizing for $\boptenergy[k,\infty](\mG)$, i.e. $\benergy[k,\infty](\mathcal P) = \boptenergy[k,\infty](\mG)$. In this case, if $\hat{\mathcal P}=(\hat \mV_1, \ldots, \hat \mV_k)$ is another $k$-partition of $\mG$ for which at least one induced subgraph $\hat\mG_i$ is not a path of length $n$, then $\hat{\mathcal P}$ is not minimizing. 
\end{lemma}
\begin{proof}
    Take $\mathcal P$ as described in the statement, and let $\hat{\mathcal P} =(\hat \mV_1, \ldots, \hat \mV_k)$ be any other $k$-partition on $\mG$ which is not of this form. Then by the pigeonhole principle either $|\mV_i|=n$ for each $i=1,\ldots, k$ or there exists $j=1,\ldots,k$ such that $|\mV_j| \le n-1$.
    
    Suppose the latter, then by monotonicity of $\llB (\mP_n)=4\sin^2\left(\frac{\pi}{2n}\right)$ (see~\cite[Section~5]{AndMor71}) with respect to $n$ and by Fiedler's inequality (see \cite[4.3]{Fie73}, noting that $\llB(\hat\mV_i)$ corresponds to the algebraic connectivity of the corresponding induced subgraph $\hat\mG_i$),
    \begin{gather*}
        \benergy[k,\infty](\hat{\mathcal P})\ge \llB(\hat\mV_j) \ge \llB(\mP_{|\hat\mV_j|}) > \llB(\mP_n) = \benergy[k,\infty](\mathcal P).
    \end{gather*}

    Suppose now that $|\hat \mV_i|=n$ for all $i=1,\ldots,k$; then not all $\hat\mG_i$ are path graphs. Suppose $\hat\mG_j$ is not a path. Then, by the characterization of equality in Fiedler's inequality (see \cite[Theorem~4.5]{Fie89}), $\llB(\hat\mV_j) > \llB(\mP_n)$, whence again $\benergy[k,\infty](\hat{\mathcal P}) > \benergy[k,\infty](\mathcal P)$.

    We conclude that $\mathcal P$ is minimizing $\boptenergy[k,\infty](\mG)$, and no other kind of partition which does not consist of $k$ path graphs on $n$ vertices each, can be minimizing.
\end{proof}

We next give an auxiliary result which will allow us to analyze the Dirichlet case.

\begin{lemma}\label{lem:dirichletbox}
Let $\mG_{\ell,m}$ be as in \autoref{def:ladder-clusters}, with $\ell \geq 0$, and assume its boundary to consist of the vertices $\mv_{2,\ell},\mw_1,\ldots,\mw_m$ (using the labeling scheme depicted in \autoref{fig:ladder-clusters},with the vertex $\mv_{2,\ell}$ being counted if and only if $\ell \geq 1$). Then for $m \ge 2$ we have
\begin{equation*}
    \llD(\mG_{\ell,m}) > \llD(\mG_{\ell+1,m-2}).
\end{equation*}
\end{lemma}

\begin{proof}
The proof is based on a rearrangement argument using the variational characterization of $\llD(\mG_{\ell,m})$. By the Perron--Frobenius theorem, the eigenvector corresponding to $\llD(\mG_{\ell,m})$, unique up to scalar multiples, can be chosen strictly positive on each vertex of $\mG_{\ell,m}$; we fix such a choice of eigenvector $\psi \neq 0$.

Suppose that $\ell \geq 1$. Denote the value it takes at each vertex by $a_1,\ldots,a_\ell,b_1,\ldots,b_m > 0$ on $\mP_\ell \times \mP_2$, and by $a_{\ell+1},\ldots,a_{\ell+m}>0$ the values it takes on the pendant path $\mP_m$ (see \autoref{fig:rearrangement}).

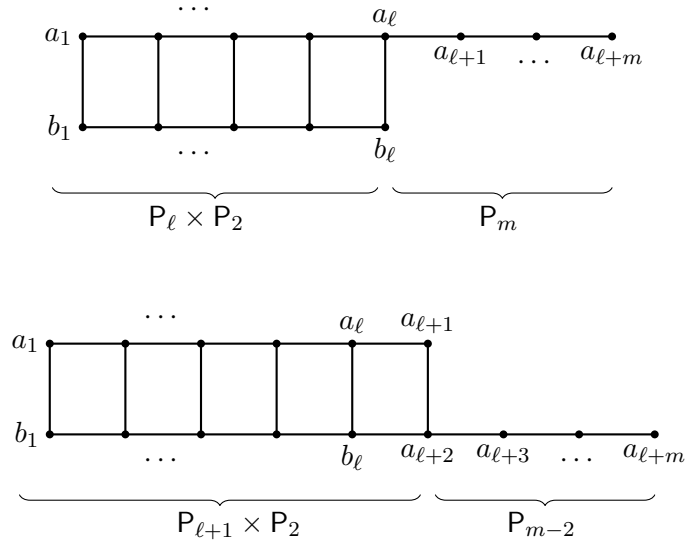
\begin{figure}[ht]
\begin{tikzpicture}[scale=1, line cap=round, line join=round]
  \def\h{1.2}
  \def\n{4}

  \foreach \i in {0,...,\n} {
    \fill (\i,0) circle (1.5pt);
    \fill (\i,\h) circle (1.5pt);
  }

  \foreach \i in {0,...,\numexpr\n-1\relax} {
    \draw[thick] (\i,0) -- (\i+1,0);
    \draw[thick] (\i,\h) -- (\i+1,\h);
  }

  \foreach \i in {0,...,\n} {
    \draw[thick] (\i,0) -- (\i,\h);
  }

  \draw[thick] (\n,\h) -- (\n+1,\h) -- (\n+2,\h) -- (\n+3,\h);

  \fill (\n+1,\h) circle (1.5pt);
  \fill (\n+2,\h) circle (1.5pt);
  \fill (\n+3,\h) circle (1.5pt);

  \node at (1.5,-0.35) {$\cdots$};
  \node at (1.5,\h+0.35) {$\cdots$};

  \node at (\n+2,\h-0.35) {$\cdots$};

  \node[left] at (0,0) {$b_1$};
  \node[left] at (0,\h) {$a_1$};
  \node[below] at (\n,0) {$b_\ell$};
  \node[above] at (\n,\h) {$a_\ell$};
  \node[below] at (\n+3,\h) {$a_{\ell+m}$};

  \node[below] at (\n+1,\h) {$a_{\ell+1}$};

  \draw[decorate,decoration={brace,amplitude=4pt,mirror}]
    (-0.4,-0.8) -- (\n-0.1,-0.8);
  \node at (1.5,-1.2) {$\mP_\ell \times \mP_2$};

  \draw[decorate,decoration={brace,amplitude=4pt,mirror}]
    (\n+0.1,-0.8) -- (\n+3,-0.8);
  \node at (5.5,-1.2) {$\mP_m$};
\end{tikzpicture} \vspace{2 em} \\

\begin{tikzpicture}[scale=1, line cap=round, line join=round]
  \def\h{1.2}
  \def\n{5}

  \foreach \i in {0,...,\n} {
    \fill (\i,0) circle (1.5pt);
    \fill (\i,\h) circle (1.5pt);
  }

  \foreach \i in {0,...,\numexpr\n-1\relax} {
    \draw[thick] (\i,0) -- (\i+1,0);
    \draw[thick] (\i,\h) -- (\i+1,\h);
  }

  \foreach \i in {0,...,\n} {
    \draw[thick] (\i,0) -- (\i,\h);
  }

  \draw[thick] (\n,0) -- (\n+1,0) -- (\n+2,0) -- (\n+3,0);

  \fill (\n+1,0) circle (1.5pt);
  \fill (\n+2,0) circle (1.5pt);
  \fill (\n+3,0) circle (1.5pt);

  \node at (1.5,-0.35) {$\cdots$};
  \node at (1.5,\h+0.35) {$\cdots$};

  \node at (\n+2,-0.35) {$\cdots$};

  \node[left] at (0,0) {$b_1$};
  \node[left] at (0,\h) {$a_1$};
  \node[below] at (\n-1, 0){$b_\ell$};
  \node[above] at (\n-1, \h){$a_\ell$};
  \node[below] at (\n,0) {$a_{\ell+2}$};
  \node[above] at (\n,\h) {$a_{\ell+1}$};
  \node[below] at (\n+3,0) {$a_{\ell+m}$};
  \node[below] at (\n+1, 0) {$a_{\ell+3}$};

  \draw[decorate,decoration={brace,amplitude=4pt,mirror}]
    (-0.4,-0.8) -- (\n-0.1,-0.8);
  \node at (2.5,-1.2) {$\mP_{\ell+1} \times \mP_2$};

  \draw[decorate,decoration={brace,amplitude=4pt,mirror}]
    (\n+0.1,-0.8) -- (\n+3,-0.8);
  \node at (6.5,-1.2) {$\mP_{m-2}$};
\end{tikzpicture} 
\caption{Rearrangement concept in the proof of \autoref{lem:dirichletbox}. Above: the starting graph $\mG_{\ell,m}$, together with the values $a_1,\ldots,a_{\ell+m}, b_1, \ldots, b_\ell$ that the given eigenvector $\psi$ takes at each of the vertices of $\mG_{\ell,m}$. Below: the comparison graph $\mG_{\ell+1,m-2}$, together with the rearranged values of the test vector on $\mG_{\ell+1,m-2}$, constructed using the values of $\psi$.}
\label{fig:rearrangement}
\end{figure}

We rearrange the values of $\psi$ as illustrated in \autoref{fig:rearrangement}: the subpath starting at $a_{\ell+2}$ is reattached below the vertex corresponding to $a_{\ell+1}$. In particular, after rearrangement, the vertex with value $a_{\ell+2}$ becomes adjacent to the vertex with value $b_\ell$ on $\mP_\ell \times \mP_2$.

We claim that the resulting vector $\tilde\psi$ on $\mG_{\ell+1,m-2}$ has a strictly smaller Rayleigh quotient than $\psi$. Indeed, the $\ell^2$-norm of the two vectors is equal, since they take on exactly the same values; while the local contribution to the Dirichlet energy decreases, since
\[
(a_{\ell+2} - b_\ell)^2 < (a_{\ell+2} - 0)^2 + (b_\ell - 0)^2
\]
(where in all other cases the difference in neighboring values remains the same). Denoting by $R[\tilde\psi]$ the corresponding Rayleigh quotient, we thus have that
\begin{displaymath}
    \llD (\mG_{\ell+1,m-2}) \leq R[\tilde\psi] < R[\psi] = \llD (\mG_{\ell,m}).
\end{displaymath}
The case $\ell = 0$ is an easy adaptation of the above argument, and is omitted.
\end{proof}

\begin{cor}\label{cor:optimal-constellation}
    Let $m, \ell \in \mathbb N_0$ and $N\in \mathbb N$, such that $N=2\ell + m$.
    \begin{itemize} 
    \item[(i)] If $N=2n+1$ for some $n\in \mathbb N$, then
    \begin{gather*}
        \llD(\mG_{\ell,m}) \ge \llD(\mG_{n,1}),
    \end{gather*}
    with equality if and only if $m=1$ and $\ell=n$.
    \item[(ii)] If $N=2n$ for some $n\in \mathbb N$, then
    \begin{gather*}
        \llD(\mG_{\ell,m}) \ge \llD(\mG_{n,0}) = \llD(\mLn),
    \end{gather*}
    with equality if and only if $m=0$ and $\ell=n$.
    \end{itemize}
\end{cor}

\begin{figure}[ht]
\centering
    \begin{tikzpicture}[scale=1, line cap=round, line join=round]
  \def\n{3} 
  \def\h{1.2}

  \foreach \i in {0,...,\n} {
    \fill (\i,0) circle (1.5pt);
    \fill (\i,\h) circle (1.5pt);
  }

  \foreach \i in {0,...,\numexpr\n-1\relax} {
    \draw[thick] (\i,0) -- (\i+1,0);
    \draw[thick] (\i,\h) -- (\i+1,\h);
  }

  \foreach \i in {0,...,\n} {
    \draw[thick] (\i,0) -- (\i,\h);
  }

  \draw[dashed] (\n, 0) -- (\n +1, 0);
  \draw[dashed] (\n, \h) -- (\n+1, \h);
  \draw[thick] (\n+1, 0) -- (\n+1, \h);

\fill (\n+1,0) circle (1.5pt);
\fill (\n+1, \h) circle (1.5pt);

  \node[left] at (0,-0.4) {$b_1$};
  \node[left] at (0,\h+0.1) {$a_1$};
  \node[right] at (\n+1,\h+0.1) {$a_{n}$};
  \node[right] at (\n+1,-0.4) {$b_{n}$};
\end{tikzpicture}
\qquad     \begin{tikzpicture}[scale=1, line cap=round, line join=round]
  \def\n{3} 
  \def\h{1.2}

  \foreach \i in {0,...,\n} {
    \fill (\i,0) circle (1.5pt);
    \fill (\i,\h) circle (1.5pt);
  }
  \fill (\n+1, \h) circle (1.5pt);
  \fill (\n+2, \h) circle (1.5pt);

  \foreach \i in {0,...,\numexpr\n-1\relax} {
    \draw[thick] (\i,0) -- (\i+1,0);
    \draw[thick] (\i,\h) -- (\i+1,\h);
  }

  \foreach \i in {0,...,\n} {
    \draw[thick] (\i,0) -- (\i,\h);
  }

  \draw[dashed] (\n, 0) -- (\n +1, 0);
  \draw[dashed] (\n, \h) -- (\n+1, \h);
  \draw[thick] (\n+1, \h) -- (\n+2, \h);
  \draw[thick] (\n+1, 0) -- (\n+1, \h);

\fill (\n+1,0) circle (1.5pt);
\fill (\n+1, \h) circle (1.5pt);

  \node[left] at (0,-0.4) {$b_1$};
  \node[left] at (0,\h+0.1) {$a_1$};
  \node[above] at (\n+1,\h+0.1) {$a_{n}$};
  \node[right] at (\n+1,-0.4) {$b_{n}$};
  \node[right] at (\n+2, \h+0.1) {$a_{n+1}$};
\end{tikzpicture}

\caption{The ladder $\mLn= \mG_{n,0} = \mP_n \times \mP_2$ ($N$ even) and $\mG_{n,1}$ ($N$ odd) are optimal in the sense of \autoref{cor:optimal-constellation}.}
\label{fig:two-step-ladder}
\end{figure}

We can now characterize all the spectral minimal $2$-partitions on $\mLn = \mP_n \times \mP_2$, $n \geq 2$, for the three energy functionals. We note that valid $2$-partitions, since they must be exhaustive, can only consist of cells of the form $\mG_{\ell,m}$, for $0 \leq m \leq n$ and $0 \leq \ell \leq \lfloor \frac{n}{2}\rfloor$.

\begin{itemize}
    \item Then, by \autoref{cor:optimal-constellation}, the unique spectral minimal partition in the Dirichlet sense depends on whether $n$ is even or odd. If $n=2m$ with $m\in \mathbb N$, then the spectral minimal partition consists of two copies of $\mG_{m,0} = \mL_m$. If $n=2m+1$, $m \in \N$, then the spectral minimal partition consists of two copies of $\mG_{m,1}$. See \autoref{fig:ladder-optimal-dirichlet}. 
    \item The spectral minimal partition in the boundaryless case consists of a division of $\mLn$ into two path graphs of $n$ vertices each; since there exist such partitions, they must be the only minimizers, by \autoref{lem:optimal-config-boundaryless}. Note, however, that there are two possible configurations: Either one cuts along the horizontal lines in $\mLn$ to create a partition consisting of two copies of $\mP_m$, or for $m\ge 2$ one can also divide the graph in two path graphs in an $L$-shape, where each cell is of the form $\mG_{1,m-2}$. See \autoref{fig:ladder-optimal-boundaryless}.
    
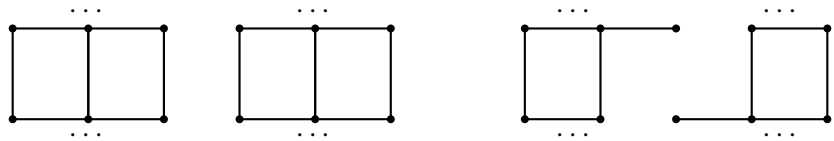
\begin{figure}[ht]
        \begin{tikzpicture}[scale=1, line cap=round, line join=round]
  \def\h{1.2}
  \def\n{2}
  \pgfmathsetmacro{\m}{2*\n+1}

  \foreach \i in {0,...,\m} {
    \fill (\i,0) circle (1.5pt);
    \fill (\i,\h) circle (1.5pt);
  }

  \foreach \i in {0,...,\numexpr\n-1\relax} {
    \draw[thick] (\i,0) -- (\i+1,0);
    \draw[thick] (\i,\h) -- (\i+1,\h);
    \draw[thick] (\i, 0) -- (\i, \h);
    \draw[thick] (\i+1, 0) -- (\i+1, \h);
  }

    \foreach \i in {\numexpr\n+1\relax,...,\numexpr2*\n\relax} {
    \draw[thick] (\i,0) -- (\i+1,0);
    \draw[thick] (\i,\h) -- (\i+1,\h);
    \draw[thick] (\i, 0) -- (\i, \h);
    \draw[thick] (\i+1, 0) -- (\i+1, \h);
  }

  \node[above] at (\n/2, \h) {$\cdots$};
  \node[below] at (\n/2,0) {$\cdots$};

  \node[above] at (2*\n, \h) {$\cdots$};
  \node[below] at (2*\n,0) {$\cdots$};
  
  \end{tikzpicture} \qquad \qquad 
          \begin{tikzpicture}[scale=1, line cap=round, line join=round]
  \def\h{1.2}
  \def\n{2}
  \pgfmathsetmacro{\m}{2*\n}

  \foreach \i in {0,...,\m} {
    \fill (\i,0) circle (1.5pt);
    \fill (\i,\h) circle (1.5pt);
  }

  \foreach \i in {0,...,\numexpr\n-2\relax} {
    \draw[thick] (\i,0) -- (\i+1,0);
    \draw[thick] (\i,\h) -- (\i+1,\h);
    \draw[thick] (\i, 0) -- (\i, \h);
    \draw[thick] (\i+1, 0) -- (\i+1, \h);
  }
  \draw[thick] (\n, 0) -- (\n+1, 0);
  \draw[thick] (\n, \h) -- (\n-1, \h);

    \foreach \i in {\numexpr\n+1\relax,...,\numexpr2*\n-1\relax} {
    \draw[thick] (\i,0) -- (\i+1,0);
    \draw[thick] (\i,\h) -- (\i+1,\h);
    \draw[thick] (\i, 0) -- (\i, \h);
    \draw[thick] (\i+1, 0) -- (\i+1, \h);
  }

  \node[above] at (\n/3, \h) {$\cdots$};
  \node[below] at (\n/3,0) {$\cdots$};

  \node[above] at (1.7*\n, \h) {$\cdots$};
  \node[below] at (1.7*\n,0) {$\cdots$};
  
  \end{tikzpicture} 
    \caption{The partitions consisting of the arrangements in \autoref{fig:two-step-ladder} are optimal in the Dirichlet case.}
    \label{fig:ladder-optimal-dirichlet}
\end{figure}

    \begin{figure}
    \begin{tikzpicture}[scale=1, line cap=round, line join=round]
  \def\h{1.2}
  \def\n{5}

  \foreach \i in {0,...,\n} {
    \fill (\i,0) circle (1.5pt);
    \fill (\i,\h) circle (1.5pt);
  }

  \foreach \i in {0,...,\numexpr\n-1\relax} {
    \draw[thick] (\i,0) -- (\i+1,0);
    \draw[thick] (\i,\h) -- (\i+1,\h);
  }

  \node[above] at (2.5, \h) {$\cdots$};
  \node[below] at (2.5,0) {$\cdots$};
  
  \end{tikzpicture}
  \qquad  \qquad    
  \begin{tikzpicture}[scale=1, line cap=round, line join=round]
  \def\h{1.2}
  \def\n{5}

  \foreach \i in {0,...,\n} {
    \fill (\i,0) circle (1.5pt);
    \fill (\i,\h) circle (1.5pt);
  }

  \foreach \i in {0,...,\numexpr\n-2\relax} {
    \draw[thick] (\i,0) -- (\i+1,0);
    \draw[thick] (\i+1,\h) -- (\i+2,\h);
  }
  \draw[thick] (0,0) -- (0, \h);
  \draw[thick] (\n,0) -- (\n, \h);

  \node[above] at (2.5, \h) {$\cdots$};
  \node[below] at (2.5,0) {$\cdots$};
  
  \end{tikzpicture}
  \caption{The partitions consisting of two path graphs or two $L$-shaped graphs provide the spectral minimal $2$-partition in the boundaryless case.}
  \label{fig:ladder-optimal-boundaryless}
\end{figure}
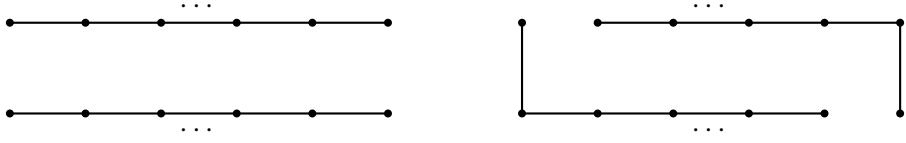
    \item In the Neumann case, we have $\noptenergy[k,\infty](\mLn) =\boptenergy[k,\infty](\mLn)$, since on the one hand $\noptenergy[k,\infty](\mG) \geq \boptenergy[k,\infty](\mG)$ in general, see \autoref{prop:comparison-bn}, but on the other, it can be seen directly that $\llN(\mG_{0,n})=\llB(\mG_{0,n})$, that is, in $\mL_n$, the horizontal path graphs $\mP_n$ have the same Neumann as boundaryless eigenvalue, and thus continue to form a minimal partition. However, in the Neumann case, this is the \emph{only} minimal partition. Indeed, by \cite[Theorem~1.1]{ShiYu25}, $\llN(\mG_{1,n-2}) > \llB(\mG_{1,n-2})= \llB(\mP_n)$ since an eigenvector of the Neumann problem on the L-shaped subgraph $\mG_{1,n-2}$ cannot coincide with the eigenvector in the boundaryless case due to the existence of a vertex in the vertex boundary $\delta \mV_{1,n-2}$ that is adjacent to more than a single vertex of the subgraph $\mG_{1,n-2}$.
\end{itemize}

We leave it as an open question whether graphs $\mG$ exist such that, for some $k$ and some $p$, $\noptenergy[k,p](\mG)$ is attained on a partition that is different from any minimizer of both $\doptenergy[k,p](\mG)$ and $\boptenergy[k,p](\mG)$.

\subsection{The infinite case} We now consider the case $n=\infty$. We first observe that under our standing assumptions the partition problems are always well posed.

\begin{prop}
    Take $\mL_\infty = \Z \times \mP_2$ to be the infinite two-ladder. If $\nu$ is finite, that is, $\nu (\mV(\mL_\infty)) < \infty$, and $L(\mL_\infty) < \infty$ (see \autoref{cor:geometric-condition-for-ucc}), then for all $k \geq 2$ and all $1 \leq p \leq \infty$ the functionals $\doptenergy[k,p](\mL_\infty)$, $\boptenergy[k,p](\mL_\infty)$ and $\noptenergy[k,p](\mL_\infty)$ each admit a minimizing $k$-partition.
\end{prop}

\begin{proof}
    The requirements of \autoref{thm:dir-existence} apply, hence, the spectral minimizing $k$-partitions exist for all $k\ge 2$. In order to apply \autoref{thm:existence-bdryless} and \autoref{thm:existence-neumann} we need to verify the requirements of \ autoref {lem: classconn-closed} additionally. In fact, there exist in
    $\mL_\infty = \mathbb Z \times \mP_2$ only finitely many non-contractible paths between any two vertices, hence \autoref{lem:classconn-closed} applies, and each of the problems $\doptenergy[k,p](\mL_\infty)$, $\boptenergy[k,p](\mL_\infty)$ and $\noptenergy[k,p](\mL_\infty)$ admit a minimizing $k$-partition.
\end{proof}

\begin{rem}
(1)   On the other hand, if on $\mL_\infty = \Z \times \mP_2$ we take all vertex weights and all edge weights to be $1$, then the quantities \eqref{eq:llb} (boundaryless case) and \eqref{eq:lambda2-rayl-chung-2} (Neumann case) are not well-defined on any cluster containing infinitely many vertices (of which there must be at least one in any partition). Indeed, it is clear from a simple test function argument that on any such cluster the respective infima in \eqref{eq:llb} and \eqref{eq:lambda2-rayl-chung-2} are both zero; however they will not be attained, that is, there is no minimum (since for the energy to be zero, the function must be constant, hence not square summable).
    
(2)    In the Dirichlet case, the problem is similar: on an infinite graph 
    one can still define $\llD$ to be the infimum of the spectrum; however, this infimum may be zero and may not be attained by any element of $\ell^2$. But there is a further problem: if we accept this broader definition of $\llD$, then on $\mL_\infty$ we will have that $\doptenergy[k,p](\mL_\infty) = \inf \denergy[k,p](\mathcal P) = 0$ for all $k \geq 2$ and all $1 \leq p \leq \infty$, but the infimum over all partitions, $0$, will only be attained for $k=2$.
    
    Namely, suppose for a contradiction that there exists a $k$-partition attaining $\doptenergy[k,p](\mL_\infty)$ for $k\ge 3$, then at least one cluster will either be finite or consist
    
    of a semi-infinite path graph, by which we mean a path graph of infinite length with one degree-one vertex.  However, the first eigenvalue of any finite set is strictly positive, while for any semi-infinite path graph $\mP$ we have $\degout(\mw)\ge 1$  (where $\degout$  was defined in \eqref{eq:degout-def}) for each vertex in $\mP$. By the variational characterization we have $$\lambda_1^{\mathrm D}(\mP_\infty )\ge \min_{\mw\in \mP_\infty} \degout(w) \ge 1. $$
    Since the infimum over all partitions is clearly zero (take the partition $\mathcal P$ to consist of two semi-infinite two-ladder graphs, of the form $\mG_{\infty,0}:= \mathbb N \times \mP_2$, separated by $k-2$ finite ladders $\mL_n$ each of depth $n$: then $\llD(\mL_n)\searrow 0$ as $n\to \infty$, and so does $\denergy[k,p](\parti)$), no partition for $k\ge 3$ can attain the spectral spectral minimal energy $\doptenergy[k,1](\mL_\infty)=0$. This phenomenon is reminiscent of the case of spectral minimal partitions of infinite metric graphs (see \cite[Example~5.1]{HoKeSe23}). 

    For $k=2$, spectral minimal partitions exist (in the generalized sense where $\llD$ is an infimum) for every $1 \le p \le \infty$, but are not unique; again, cf.\ \cite[Example~5.1]{HoKeSe23}. In fact, any $2$-partition is minimal if each of its two clusters contains a semi-infinite ladder graph $\mG_{\infty,0}$ as a subgraph; in our running notation, such clusters may be denoted by $\mG_{\infty,m}$: for each $m \in \N_0$ we will have a different minimizing partition.
    
    To see that such partitions are minimal, by domain monotonicity for each of the induced subgraphs $\mG_{\infty,m}$ of an element of the partition, we then infer
    \begin{gather*}
        \llD(\mG_{\infty,m}) \le \llD(\mG_{\infty,0}) \le \lim_{n\to \infty} \frac{4}{2n}=0,
    \end{gather*}
    {
    where, in the variational argument, we choose characteristic functions on embedded subgraphs $\mG_{n,0}$ in $\mL_\infty$ as test functions.}

\end{rem}

\newcommand{\etalchar}[1]{$^{#1}$}

\end{document}